\documentclass[12pt]{amsart}
\usepackage{verbatim,amssymb,latexsym,amscd}
\usepackage[all]{xy}
\usepackage[colorlinks,linkcolor=blue,citecolor=blue,urlcolor=red]{hyperref}
\usepackage[T1]{fontenc}

\newcommand{\Cnd}{\operatorname{Cnd}}
\newcommand{\ie}{{\it i.e. }}
\newcommand{\cf}{{\it cf. }}
\newcommand{\eg}{{\it e.g. }}
\newcommand{\loccit}{{\it loc. cit. }}
\newcommand{\resp}{{\it resp. }}
\newcommand{\un}{\mathbf{1}}
\newcommand{\A}{\mathbf{A}}

\newcommand{\C}{\mathbf{C}}

\newcommand{\G}{\mathbb{G}}
\newcommand{\F}{\mathbf{F}}
\renewcommand{\H}{\mathbf{H}}
\renewcommand{\L}{\mathbb{L}}

\renewcommand{\P}{\mathbf{P}}
\newcommand{\Q}{\mathbf{Q}}
\newcommand{\R}{\mathbf{R}}

\newcommand{\Z}{\mathbf{Z}}
\newcommand{\sA}{\mathcal{A}}

\newcommand{\sE}{\mathcal{E}}

\newcommand{\sM}{\mathcal{M}}
\newcommand{\sN}{\mathcal{N}}

\newcommand{\sV}{\mathcal{V}}
\newcommand{\sX}{\mathcal{X}}
\newcommand{\sY}{\mathcal{Y}}
\newcommand{\sZ}{\mathcal{Z}}

\newcommand{\Tr}{\operatorname{Tr}}
\newcommand{\card}{\operatorname{card}}
\newcommand{\Spec}{\operatorname{Spec}}

\newcommand{\Ker}{\operatorname{Ker}}
\newcommand{\Coker}{\operatorname{Coker}}
\newcommand{\IM}{\operatorname{Im}}

\newcommand{\red}{{\operatorname{red}}}

\newcommand{\PHS}{\operatorname{\bf PHS}}
\newcommand{\MT}{\operatorname{MT}}
\newcommand{\Gal}{\operatorname{Gal}}
\newcommand{\Sp}{\operatorname{\mathbf{Sp}}}
\newcommand{\GL}{\operatorname{\mathbf{GL}}}
\newcommand{\Lie}{\operatorname{\mathbf{Lie}}}
\newcommand{\rank}{\operatorname{rk}}

\newcommand{\tr}{{\operatorname{tr}}}

\newcommand{\Mor}{\operatorname{Mor}}

\newcommand{\Hom}{\operatorname{Hom}}
\newcommand{\End}{\operatorname{End}}

\newcommand{\trdeg}{\operatorname{trdeg}}
\newcommand{\Alb}{\operatorname{Alb}}
\newcommand{\NS}{\operatorname{NS}}

\newcommand{\Pic}{\operatorname{Pic}}
\newcommand{\Griff}{\operatorname{Griff}}

\newcommand{\cl}{{\operatorname{cl}}}
\newcommand{\Id}{{\operatorname{Id}}}
\newcommand{\Nis}{{\operatorname{Nis}}}

\newcommand{\uHom}{\operatorname{\underline{Hom}}}
\newcommand{\Ext}{\operatorname{Ext}}
\renewcommand{\Vec}{{\operatorname{\bf Vec}}}

\newcommand{\AbS}{\operatorname{\bf AbS}}
\newcommand{\Ab}{\operatorname{\bf Ab}}

\newcommand{\Loc}{\operatorname{\bf Loc}}
\newcommand{\Mot}{\operatorname{\bf Mot}}
\newcommand{\MHS}{\operatorname{\bf MHS}}

\newcommand{\diag}{\operatorname{diag}}
\newcommand{\car}{\operatorname{char}}

\newcommand{\rat}{{\operatorname{rat}}}
\newcommand{\alg}{{\operatorname{alg}}}
\renewcommand{\hom}{{\operatorname{hom}}}
\newcommand{\num}{{\operatorname{num}}}
\newcommand{\mot}{{\operatorname{mot}}}

\newcommand{\DM}{\operatorname{\bf DM}}
\newcommand{\Sm}{\operatorname{\bf Sm}}

\newcommand{\ab}{{\operatorname{Ab}}}
\newcommand{\abb}{{\operatorname{ab}}}
\newcommand{\prim}{{\operatorname{prim}}}
\newcommand{\proj}{{\operatorname{proj}}}

\newcommand{\Exc}{\operatorname{Exc}}

\newcommand{\Chow}{\operatorname{\bf Chow}}
\newcommand{\gm}{{\text{\rm gm}}}

\newcommand{\et}{{\text{\rm \'et}}}
\renewcommand{\o}{{\text{\rm o}}}
\newcommand{\op}{{\text{\rm op}}}
\newcommand{\eff}{{\text{\rm eff}}}

\newcommand{\by}{\xrightarrow}
\newcommand{\yb}{\xleftarrow}
\newcommand{\iso}{\by{\sim}}
\newcommand{\osi}{\yb{\sim}}
\newcommand{\inj}{\hookrightarrow}
\newcommand{\Inj}{\lhook\joinrel\longrightarrow}
\newcommand{\surj}{\rightarrow\!\!\!\!\!\rightarrow}
\newcommand{\Surj}{\relbar\joinrel\surj} 
\newcommand{\colim}{\varinjlim}
\renewcommand{\lim}{\varprojlim}

\newcommand{\gr}{\operatorname{gr}}

\renewcommand{\qed}{\hfill $\Box$\medskip}

\renewcommand{\phi}{\varphi}
\renewcommand{\epsilon}{\varepsilon}

\newcounter{spec}
\newenvironment{thlist}{\begin{list}{\rm{(\roman{spec})}}%
{\usecounter{spec}\labelwidth=20pt\itemindent=0pt\labelsep=10pt}}%
{\end{list}}%

\numberwithin{equation}{section}

\newtheorem{Thm}{Theorem}
\newtheorem{Corr}{Corollary}
\newtheorem{thm}{Theorem}[section]
\newtheorem{lemma}[thm]{Lemma}
\newtheorem{prop}[thm]{Proposition}
\newtheorem{cor}[thm]{Corollary}
\newtheorem{conj}[thm]{Conjecture}

\theoremstyle{definition}

\newtheorem{defn}[thm]{Definition}
\newtheorem{hyp}[thm]{Hypothesis}
\newtheorem{nota}[thm]{Notation}
\newtheorem{rk}[thm]{Remark}
\newtheorem{warn}[thm]{Warning}
\newtheorem{rks}[thm]{Remarks}

\newtheorem*{Qn}{Question}

\newtheorem{exs}[thm]{Examples}

\begin{document}
\title{Albanese kernels and Griffiths groups}
\author{Bruno Kahn}
\address{IMJ-PRG\\ Case 247\\4 place Jussieu\\
75252 Paris Cedex 05\\France}
\email{bruno.kahn@imj-prg.fr}
\address{IMJ-PRG\\ Case 247\\4 place Jussieu\\
75252 Paris Cedex 05\\France}
\email{yves.andre@imj-prg.fr}
\date{July 25, 2020}
\maketitle
\hfill With an appendix by Yves André

\begin{abstract} We describe the Griffiths group of the product of a curve $C$ and a surface $S$ as a quotient of the Albanese kernel of $S$ over the function field of $C$. When $C$ is a hyperplane section of $S$ varying in a Lefschetz pencil, we prove the nonvanishing in $\Griff(C\times S)$ of a modification of the graph of the embedding $C\inj S$ for infinitely many members of the pencil, provided the ground field $k$ is of characteristic $0$, the geometric genus of $S$ is $>0$, and $k$ is large or $S$ is ``of motivated abelian type''.
\end{abstract}

\enlargethispage*{30pt}

\tableofcontents
\newpage

\section{Introduction}

\subsection{Cycles of higher codimension}\label{s1.0}  In the late sixties, the naïve expectation that Chow groups of algebraic cycles of codimension $>1$ on smooth projective varieties should have the same structure as the Picard group was destroyed by two types of counterexamples:

\begin{description}
\item[Mumford] nonvanishing of the Albanese kernel $T(S)$ for most surfaces $S$ \cite{mumford,bloch}.
\item[Griffiths] non agreement of algebraic and homological equivalence for $1$-cycles on some $3$-folds \cite{griff}.
\end{description}

Mumford's theorem is the easier to prove and the more clear-cut: $T(S)$ is large when the base field $k$ is large for any surface $S$ such that $b^2>\rho$ (where $b^2$ and $\rho$ are respectively the second Betti number and the Picard number), and there is a conjectural converse (Bloch's conjecture). We recall a simple proof of Mumford's theorem in \S \ref{s3.2}. By contrast, the nonvanishing of the Griffiths group $\Griff(T)$ of a smooth projective $3$-fold $T$ has always been hard to obtain and is limited to scattered cases, with no general picture in sight, even conjectural; perhaps a reason is that a true understanding of $\Griff(T)$ hinges on mixed motives, rather than pure motives as in the case of $T(S)$ (compare Remark \ref{r7.11}). To the best of my knowledge, $\Griff(T)$ is known not to vanish in the following cases,  over $k=\C$:

\begin{itemize}
\item Griffiths' original example \cite{griff}: the intersection in $\P^5$ of a hypersurface of degree $\ge 5$ with a generic hyperplane section. 
\item Ceresa's example \cite{ceresa}: the Jacobian of a general curve of genus $3$. Here an explicit nonzero element of the Griffiths group is given by the Ceresa cycle.
\item Schoen's examples \cite{schoen}:  Shimura $3$-folds parametrising abelian surfaces with complex multiplication from an indefinite quaternion algebra over $\Q$.
\item Albano's examples \cite{albano}: certain elliptic fibrations over $\P^2$.
\end{itemize}

These examples were then extended in three ways. First, Clemens \cite{clemens} proved that in the example of Griffiths, $\Griff(T)$ is not finitely generated (even $\otimes \Q$), and Nori \cite{nori} proved the same for the example of Ceresa; see also Bardelli \cite{bardelli} and Voisin \cite{voisin2}, who extended it to general Calabi-Yau $3$-folds in \cite{voisinCY} following earlier work of Paranjape \cite{paran}. Schoen and Albano directly prove infinite generation. We shall not touch this issue here, except in Remark \ref{r11.1}. 

Two other issues were to prove the nonvanishing of $\Griff(T)$ over other fields than $\C$, and to give explicit rather than generic examples. Grothendieck-Katz  \cite{katz} extended Griffiths' example to a (large) base field $k$ of arbitrary characteristic, while B. Harris \cite{harris1,harris2} used a different method to prove Ceresa's theorem, which allowed him to show that it holds (over $\C$) for the degree $4$ Fermat curve $C$. Inspired by  this method, Bloch \cite{bloch2} used arithmetic techniques to show that, over $\Q$, the Ceresa cycle of $C$ has infinite order in the Griffiths group of its Jacobian. For further work in this direction, see Top \cite{top}, Zelinsky \cite{zelinsky}, Tadokoro \cite{tadokoro1,tadokoro2,tadokoro3} and Otsubo \cite{otsubo}. Schoen proves that the Griffiths group of his examples is infinite-dimensional over $k=\bar \Q$. (I apologize for possibly missing important works in this attempt to a survey.)

In this paper, we do two things: 1) relate the Albanese kernel of a surface $S$ over the function field of a curve $C$ with $\Griff(C\times S)$; 2) at the prompting of Spencer Bloch, prove the nonvanishing of $\Griff(C\times S)$ in many cases, when $C$ is a smooth hyperplane section of $S$. For 1), we use the formalism of pure motives as in \cite{scholl} (see Appendix \ref{T1} for a direct approach), while 2) uses known results on the variations of various tannakian groups in families.

\subsection{Main results}\label{s1.2} Let $C$ (resp. $S$) be a smooth, projective, geometrically connected curve (resp. surface) over a field $k$. For any extension $E/k$, write $T(S_E)\subset CH_0(S_E)$ for the Albanese kernel of $S$ over $E$. We work in the category $\Ab\otimes \Q$ of abelian groups up to isogeny, i.e. the localisation of the category $\Ab$ of abelian groups by the Serre subcategory of abelian groups of finite exponent; see \cite[Rem. 4.13]{cycletale},  \cite[\S 2]{indec} and \cite[App. B]{bvk}.

\begin{Thm}\label{T1} There is a natural surjection in $\Ab\otimes \Q$
\begin{equation}\label{eqT0}
T(S_{k(C)})/T(S)\Surj \Griff(C\times S)
\end{equation}
where, for any $3$-fold $X$, $\Griff(X)$ denotes the group of numerically trivial cycles of codimension $2$ on $X$, modulo algebraic equivalence.
\end{Thm}

(See Remark \ref{p9.1} b) for a conjectural description of the kernel of \eqref{eqT0} as the rational points of a certain abelian variety.)

Note that this numerical Griffiths group coincides in $\Ab\otimes \Q$ with the usual one for $3$-folds in characteristic $0$, thanks to Lieberman's results \cite[Cor. 1]{lieberman}; in positive characteristic and for $l$-adic cohomology, Lieberman's arguments go through when the Tate conjecture is known.

We give two proofs of Theorem \ref{T1}. The first one is motivic and based on an extension of the computations of \cite{kmp} to $3$-folds. Namely, it is a consequence of the formulas in $\Ab\otimes \Q$:
\begin{align*}
T(S_{k(C)})/T(S) &\simeq \sM(h_1(C),t_2(S))\tag{\ref{eq0.2}}\\
\Griff(C\times S)&\simeq \sM_\alg(h_1(C),t_2(S))\tag{\ref{eq2.1}}
\end{align*}
where $\sM$ (resp. $\sM_\alg$) denotes the category of Chow motives (resp. of motives modulo algebraic equivalence), and $t_2(S)$ is the transcendental part of the motive of $S$ which was studied in \cite{kmp}. 

The second proof is essentially by hand and does not assume $\dim S=2$, see Appendix \ref{A}. This proof gives more careful control on the integers which are implicit in Theorem \ref{T1}. 
As a byproduct of this second proof, we get:

\begin{Corr}\label{C2} Suppose that $k$ is a finite field and that $S_{\bar k}$ is of abelian type: its Chow motive is a direct summand of the motive of an abelian variety. Then the group $T(S_{k(C)})$ is finite.
\end{Corr}

For a surface, to be of abelian type is a birational invariant. Examples include abelian surfaces, products of two curves and Fermat surfaces (Katsura-Shioda).

For a general $k$, we also have:

\begin{Corr}\label{C1} Let $\Omega$ be a universal domain over $k$. If $T(S_\Omega)=0$, then the abelian group $\Griff(C\times S)$ has finite exponent for any $C$. (If $k$ is a universal domain over its prime subfield, the condition $T(S)=0$ is sufficient.)
\end{Corr}

 In Corollary \ref{cA.3}, we show that one can choose the finite exponent independent of $k$, and dependent only on the index of $C$ (i.e., the gcd of the degrees of its closed points).

If $k$ is algebraically closed,  
Corollary \ref{C1} is a special case of \cite[Th. 1 (ii)]{bs}. 
Thus, Corollary \ref{C1} is not really new. The converse question: \emph{when is $\Griff(C\times S)$ non torsion?} is more interesting. We tackle it in the following special case:

Suppose $C$ is a smooth hyperplane section of $S$, and let $\psi:C\inj S$ be the corresponding closed immersion. Modifying the graph of $\psi$ appropriately, we get a class
\begin{equation*}
\psi_\#\in T(S_{k(C)})/T(S)\otimes \Q\tag{\ref{eq9.5}}
\end{equation*}
hence also a class in $\Griff(C\times S)\otimes \Q$, in view of \eqref{eqT0}.

\begin{Thm}\label{T2} Suppose $\car k=0$ and  $p_g>0$. Let $f:\tilde S\to \P^1$ be a Lefschetz pencil of hyperplane sections of $S$, and let $U\subset \P^1$ be the open subset over which $f$ is smooth. For $u\in U$, write $\psi_u:C_u=f^{-1}(u)\inj S$ for the closed immersion of the corresponding fibre, and write $\Cnd(u)$ for the condition
\[(\psi_u)_\#\ne  0 \text{ in } \Griff(C_u\times_{k(u)} S_{k(u)})\otimes \Q.\]
 Then:
\begin{thlist}
\item (Corollary \ref{c11.1}) $\Cnd(\eta)$ holds, where $\eta$ is the generic point of $U$.
\item (Theorem \ref{t12.6}) If $k=\C$,  the set of $u\in U(k)$ such that $\Cnd(u)$ holds is uncountable.
\end{thlist}
\end{Thm}

For each $u$ for which we can prove that $\Cnd(u)$ holds, a by-product of the proof yields that the motivic intermediate Jacobian of $C\times S$ defined in Lemma \ref{l2.2} has no exceptional summand  (see Definition \ref{d8.1} and Theorem \ref{t10.1} d)), and that the  $l$-adic representation $H^2_{c,l}((S_{k(u)}-C_u)_{\overline{k(u)}})$ is ``genuinely mixed'', see Proposition \ref{l9.3}: this may be of independent interest. It is likely that a similar argument yields that the Hodge structure $H^2_{c,B}((S_{k(u)}-C_u)_{\C})$ is also ``genuinely mixed'' when $k$ is a subfield of $\C$: we leave to the reader the pleasure to fill in details.

If one believes in Bloch's conjecture, Corollary \ref{C1} implies that the hypothesis $p_g>0$ is necessary in Theorem \ref{T2}. Under a stronger hypothesis on $S$, we get a stronger result (see Theorem \ref{t11.3} for a sharper statement):

\begin{Thm}\label{T3} Suppose that $S_{\bar k}$ is of motivated abelian type (Definition \ref{d11.2}). If $k$ is finitely generated over $\Q$ (resp. in general),  the set of $u\in U(k)$ such that $\Cnd(u)$ fails is finite (resp. countable).
\end{Thm}

Examples include: (a) the examples after Corollary \ref{C2},  (b) K3 surfaces, (c) surfaces of general type verifying $p_g=K^2=1$. See Examples \ref{e12.1} for more details. 

Theorems \ref{T2} and \ref{T3} prompt:

\begin{Qn} Are there examples where $\Cnd(u)$ fails?
\end{Qn}

In \cite[p. 1049]{indec}, it was imprudently asserted without proof that $\psi_\#=0$ in $\sM(h_1(C),t_2(S))$; I thank Rémi Lodh for drawing my attention to this issue. It follows from Theorems \ref{T2} and \ref{T3} that the values of $A$ and $B$ in Row $i=2$ of the bottom table of \cite[p. 1047]{indec} are not justified; the other results of \cite{indec} are not affected. It would be very interesting to understand the correct values of $A$ and $B$: by Poincaré duality, this is related to Suslin homology of the open complement. I hope to come back to this question later. 

\subsection{Comments on the proofs} In Theorem \ref{T2}, if we were only interested in the nonvanishing of $\psi_\#$ in $T(S_{k(C)})$, we could content ourselves to quote the paper of Green, Griffiths and Paranjape \cite[Th. 2]{ggp}. The added difficulty for working in $\Griff(C\times S)$ is that the Abel-Jacobi map used in \cite{ggp} does not factor through algebraic equivalence in general. This happens, however, when the first step of the coniveau filtration vanishes on the relevant cohomology group (Proposition \ref{p9.3}), and this is the extra condition that we have to achieve: it is done in Section \ref{s10}, which uses André's results on the variation of the Mumford-Tate group in families \cite{andrecomp}; the main point is Corollary \ref{c9.1}. 

Similarly, Green-Griffiths-Paranjape obtain Theorem \ref{T3} when replacing $\Griff(C\times S)$ with $T(S_{k(C)})$ and ``finite'' with ``thin'' in the sense of Serre \cite[\S 9]{mw}, while they don't need the abelian type hypothesis. For this, they use a specialisation argument due to Serre and Terasoma, which relies on Hilbert's irreducibility theorem. It turns out that this argument is not sufficient to handle the coniveau issue. For this, a decisive input is André's theory of motivated cycles \cite{incond}, notably his refinements of Deligne's theorems from \cite{ln900} and his result on the variation of the motivic Galois group in families \cite[Th. 5.2]{incond}. Recent results of Cadoret-Tamagawa on the variation of $l$-adic representations in families \cite{ct1,ct2} then allow us to refine ``thin'' to ``finite''.

Suppose $k$ finitely generated over $\Q$. It follows from André's Th. 5.2 3) in \cite{incond} that the rank of the (geometric) Néron-Severi group of a smooth projective family over $k$ is  equal to the generic rank outside a thin subset of the base. (This result goes back to Terasoma in a special case \cite{terasoma}.) Using her results with Tamagawa, Cadoret refined this, as above, by replacing ``thin'' with ``finite' \cite[Cor. 5.4]{cadoret}.  One can consider part of the present work as a higher analogue of these results, relying on  cohomological degree $3$ rather than $2$. 

Finally, while the arguments of Green-Griffiths-Paranjape also work in characteristic $p>0$ as long as $k$ is not algebraic over $\F_p$\footnote{Up to checking the correct version of Hilbert's irreducibility theorem in positive characteristic, see \cite[Ch. 12 and 13]{fj}.}, the proofs of Theorems \ref{T2} and \ref{T3}  are firmly rooted in characteristic $0$, because Hodge theory is used in an essential way in 3 different places: Theorem \ref{t1V} (the theorem of the fixed part of Steenbrink-Zucker), Lemma \ref{l9.2} (Hodge coniveau $0$ is stable under tensor product) and Lemma \ref{l11.1} (the Lefschetz (1,1) theorem). Characteristic $0$ is also used in the results of Cadoret-Tamagawa (see however recent work of Emiliano Ambrosi \cite{ambrosi}). So the issue of extending the present results to positive characteristic is interesting and open; among the uses of Hodge theory just mentioned, the least obvious to transpose to $l$-adic cohomology is Lemma \ref{l9.2}, which is \emph{false} in positive characteristic because of the supersingular elliptic curves.

In comparison, few of the examples mentioned in \S \ref{s1.0} work in positive characteristic. For Griffiths' original example, one needs primitive cohomology of the given complete intersection not to be fully of coniveau $>0$. In characteristic $0$, this follows easily from Hodge theory as in \cite[Cor. 2.8]{deligne-completes}; but in order to obtain it in characteristic $>0$ for \cite{katz}, Katz had to prove a weaker, generic result in \cite[Th. 4.1]{katz2}, whose proof uses zeta functions of varieties over finite fields. Bloch's proof for \cite{bloch2}, in turn, uses specialisations to finite fields, hence transposes naturally from number fields to global fields of positive characteristic as in Zelinsky \cite{zelinsky}. It is worth asking whether either of these methods is relevant here.

\subsection{Contents of this paper} The first 8 sections are computations in the categories of pure motives over $k$ \cite{scholl}. In \S \ref{s2}, we recall basic computations of Hom groups in these categories. \S \ref{s3} recalls some results from \cite{kmp}, including a simple proof of the theorem of Mumford and Bloch that $T(S)$ is ``huge'' for a surface $S$ with $b^2>\rho$ over a large algebraically closed field. 

\S\S \ref{s5} and \ref{s4} establish a \emph{refined Chow-Künneth decomposition}, first introduced in \cite[Th. 7.7.3]{kmp}, for $3$-folds $T$ admitting a Chow-Künneth decomposition; while its construction for surfaces in \cite{kmp} was geometric, the present one is purely category-theoretic (of course, this method also works for surfaces). (See also related papers of Vial \cite{vial,vial2}.) In particular, we have a decomposition 
\begin{equation*}
h_3(T) \simeq t_3(T) \oplus h_1(J^2)(1)\tag{\ref{eq2.5a}}
\end{equation*}
where $J^2=J^2(T)$ is a certain abelian variety and $t_3(T)$ is ``the'' transcendental part of $h_3(T)$. By construction, this decomposition is unique up to unique isomorphism modulo numerical equivalence; in particular, $J^2$ is well-defined up to isogeny. Let us stress however that, contrary to the case of surfaces,  it is by no means clear that \eqref{eq2.5a} is unique modulo \emph{rational} equivalence, even up to isomorphism (this would follow from the conjectures of Bloch-Beilinson and Murre \cite{jannsen}, see also Lemma \ref{l7.1}). 

In \S \ref{s6}, we study $J^2$. We show that it satisfies two dual universal properties (Lemma \ref{l5.1}) 
and is isogenous to Murre's algebraic representative of \cite{murreSitges} 
(Proposition \ref{p6.2}) when $k$ is perfect; see also Corollary \ref{c12.2}. 

In \S \ref{s7}, under the assumption that $T$ verifies a ``Chow-Lefschetz isomorphism'' in degree $2$, we establish the basic formula
\begin{equation*}\tag{\ref{eq7.1}}
\Griff(T)\simeq \sM_\alg(\L,t_3(T))
\end{equation*}
for any decomposition \eqref{eq2.5a}: this completes the proof of \cite[Th. 4.3.3]{birat-pure} in a slightly broader setting. By Theorem \ref{tB.1}, the Chow-Lefschetz hypothesis holds in particular for the product of a curve and a surface. We further conjecture in Conjecture \ref{co1} that $\sM(\L,t_3(T))\allowbreak\iso \sM_\alg(\L,t_3(T))$, and prove in Proposition \ref{p7.1} that this is a consequence of the Bloch-Beilin\-son--Murre conjectures (see Remark \ref{r7.11} for further comments on Conjecture  \ref{co1}). We deduce Formula \eqref{eq2.1} (see \S \ref{s1.2}) from \eqref{eq7.1} in \S \ref{s8} (the special case $T=C\times S$).

Sections \ref{s9} to \ref{s.sp} are devoted to the proofs of Theorems \ref{T2} and \ref{T3}. In Proposition \ref{p9.3}, we show that the ($l$-adic) Abel-Jacobi map factors through algebraic equivalence if one factors out the coniveau $\ge 1$ part $N^1$ of cohomology (Proposition \ref{p9.3}). \S \ref{s.hodge} contains some results on variations of (pure) Hodge structures; the most important, Corollary \ref{c9.1}, says that the generic vanishing of $N^1$ specialises outside a countable subset under suitable hypotheses. Corollary \ref{c11.1} is the generic part of Theorem \ref{T2}; we also find an intriguing surjection \eqref{eq11.8} in \S \ref{s.lef}. Section \ref{s.sp} is devoted to specialising this result: Theorem \ref{T2} (ii) is proven in \S \ref{s12.1} (Theorem \ref{t12.6}), \S \ref{s12.2} reviews André's motives and introduces the condition ``of motivated abelian type'' (Definition \ref{d11.2}), \S \ref{s12.3} proves a technical result on the coniveau filtration (Theorem \ref{t12.1}), \S \ref{s12.4} studies $l$-adic representations and draws some consequences of results of Cadoret-Tamagawa (Theorem \ref{t12.4}),  \S\ref{s12.5} proves two technical lemmas, and finally Theorem \ref{T3} is proven in \S \ref{s12.6} under a more precise form (Theorem \ref{t11.3}).  

There are 3 appendices. The first one gives a cycle-theoretic proof of Theorem \ref{T1}, relying on the study of the Albanese map on $0$-cycles initiated by Ramachandran \cite{ram} and developed in \cite{birat-pure}. We prove Corollary \ref{C2} there. The second one shows that a $3$-fold of the form $C\times S$ verifies the ``Chow-Lefschetz isomorphism in degree $2$'' alluded to above. The last one, by Yves André, gives a proof of Theorem \ref{t1V}, using (for the first time) the concept of observability in the tannakian setting.

\subsection{Acknowledgements}  I thank Claire Voisin for mentioning the article \cite{ggp} of Green-Griffiths-Paranjape in the early stages of this work. Inspiration from \cite{ggp} will be obvious in \S \ref{s.lef}; I also drew inspiration from Katz's exposé on Griffiths' theorem in SGA 7 \cite{katz}. I thank Spencer Bloch for ignoring Corollary \ref{C1}, but asking the right question which led to Theorems \ref{T2} and \ref{T3}. Thanks also to Wayne Raskind for asking a question which led to Corollary \ref{C2}.  I thank Yves André and Anna Cadoret as well for their generous help, including numerous exchanges without which completing the proofs of these theorems would have been hazardous (in a French or an English sense). I am especially grateful to Yves Andr\'e for pointing out a gap in a (perhaps too optimistic) version of Theorem \ref{t1V}, and for accepting to add an appendix to fill this gap. Finally, I thank the referee for several helpful comments, especially for pointing out a reference I had missed in \cite{jannsen2}.

\subsection{Notation}\label{s1.1} We fix the base field $k$ and denote by $\Sm^\proj(k)=\Sm^\proj$ the category of smooth projective $k$-varieties, with morphisms the morphisms of $k$-schemes. Unless otherwise specified, \emph{variety} means smooth projective variety; similarly for curve, surface, $3$-fold\dots

For any adequate equivalence relation $\sim$ on algebraic cycles \cite{samuel}, we write $\sM_\sim(k)=\sM_\sim$ for the category of effective pure motives modulo $\sim$ with $\Q$ coefficients: this is the largest full subcategory of the category of \cite{scholl}, denoted here by $\sM_\sim[\L^{-1}]$, which contains all motives of smooth projective varieties and is stable under direct sums and direct summands. We will work only occasionally in $\sM_\sim[\L^{-1}]$ to apply Poincaré duality.

We simply write $\sM$ if $\sim$ is rational equivalence, and write $\sM_H$ if $\sim$ is homological equivalence relative to a Weil cohomology $H$. If $X$ is a smooth projective variety, we write $h(X)$ for its motive in $\sM_\sim$, and only write $h_\sim(X)$ if there is a risk of confusion. We use the covariant convention, as in \cite{kmp} and unlike \cite{scholl}: the functor $h:\Sm^\proj\to \sM_\sim$ is covariant.  If $\L$ is the Lefschetz motive and $M\in \sM_\sim$, we abbreviate $M\otimes \L^{\otimes n}$ to $M(n)$ (not $M(-n)$ as in \cite{scholl}).

As in \cite[\S 2]{indec}, we use refined Hom groups with values in the category $\Ab\otimes \Q$ of abelian groups up to isogeny; by abuse, we keep the notation $\sM_\sim(-,-)$ for these refined Hom groups, and similarly for $\DM_\gm^\eff(-,-)$ in Subsection \ref{s7.4}.

For $X\in \Sm^\proj$ and $i\ge 0$, we write $A^i_\sim(X)$ (resp. $A_i^\sim(X)$) for the group of cycles of codimension (resp. dimension) $i$ on $X$, modulo $\sim$, viewed as an object of $\Ab\otimes\Q$ (see above). We shall use without further mention the formulas (opposite to \cite[2.1]{scholl}):
\begin{equation}\label{eq0}
A^i_\sim(X)=\sM_\sim(h(X),\L^i),\quad A_i^\sim(X)=\sM_\sim(\L^i,h(X)).
\end{equation}

In particular, these groups are $0$ for $i\notin[0,\dim X]$.  As in \cite[loc. cit.]{scholl}, we use these formulas to extend the functors $A^i_\sim,A_i^\sim$ to all $M\in \sM_\sim$.

A Weil cohomology is \emph{classical}  if it belongs to the following list: $l$-adic cohomology in any characteristic $\ne l$, Betti or de Rham cohomology in characteristic $0$ and crystalline cohomology in characteristic $>0$ \cite[Def. 4]{aknote}. All these cohomology theories verify 
\begin{equation}\label{eq0.1}
\dim H^1(X) = 2\dim P_X
\end{equation}
where $P_X$ is the Picard variety of $X\in \Sm^\proj$ (cf. \cite[Cor. 2A10]{kleiman-dix}). More generally, the Betti numbers of $X$ relative to $H$ do not depend on the choice of $H$: this follows from the comparison theorems in characteristic $0$ and from \cite{weilI,km} in characteristic $>0$.

We shall use repeatedly the \emph{Chow-Künneth decomposition} of Murre \cite[Conj. 5.1 (A)]{jannsen}  and the \emph{refined Chow-Künneth decomposition} of \cite[(7.12)]{kmp} (abbreviated respectively CK and refined CK decomposition). When they exist, the first one lifts the Künneth decomposition of the motive of a smooth projective variety from homological to rational equivalence, and the second one does the same with respect to a finer decomposition incorporating coniveau.

We use the notation $\approx$ to denote isogenies between abelian varieties.
\newpage

\part{Proof of Theorem \ref{T1}}

\section{Some motivic algebra}\label{s2}

Let $X\in \Sm^\proj$. By Murre \cite{murre,scholl}, we have a partial Chow-K\"unneth decomposition
\[h(X)=h_0(X)\oplus h_1(X)\oplus h_{>1}(X)\]
where $h_0(X)$ is an Artin motive and $h_1(X)$ is a direct summand of $h_1(C)$ for some smooth projective curve $C$ (which can be chosen as an iterated hyperplane section of $X$). This decomposition yields a partial K\"unneth decomposition for any Weil cohomology $H$ verifying \eqref{eq0.1} (for example, $H$ classical), because it lifts by construction Kleiman's partial K\"unneth decomposition in \cite{kleiman-dix}, which uses \eqref{eq0.1}.  We also write $h_{>0}(X)$ for $h_1(X)\oplus h_{>1}(X)$.

\begin{lemma}\label{l1}Let $\sim$ be an adequate equivalence relation.  
We have (in $\Ab\otimes \Q$)
\begin{align}
\sM_\sim(h_{>0}(Y), h_0(X))&=0\label{eq1.1}\\
\sM_\sim(\un,\un)&=\Z\text{ for any } \sim\label{eq1.3a}\\
\sM_\sim(\un,h_1(X)) &= \Alb_X(k)\text{ if } \sim=\rat\label{eq1.2}\\
\sM_\sim(\un, h_{>0}(X))&=0 \text{ if } \sim \le \alg\label{eq1.5}\\
\sM_\sim(h_1(Y),h_1(X))&=\Hom_k(\Alb_Y,\Alb_X)\text{ for any } \sim\label{eq1.3}\\
\sM_\sim(h(Y),h_1(X))&=\Mor_k(Y,\Alb_X)\text{ if } \sim=\rat \label{eq1.4a}\\
\sM_\sim(h_{>1}(Y),h_1(X))&=0.\label{eq1.4}
\end{align}
where $Y\in \Sm^\proj$.
\end{lemma}

\begin{proof} 
\eqref{eq1.1} is easy and left to the reader. \eqref{eq1.3a} is trivial. \eqref{eq1.2} and \eqref{eq1.3} are \cite[Th. 4.4 (iii)]{scholl} and \cite[Prop. 4.5]{scholl} respectively. \eqref{eq1.5} holds because $0$-cycles of degree $0$ are algebraically equivalent to $0$.

The morphism \eqref{eq5a} of Appendix A provides a map \eqref{eq1.4a}, hence for this item and \eqref{eq1.4}, we may assume $\dim X=1$, $X,Y$ connected and (by a transfer argument) even geometrically connected. Then
\[\sM(h(Y),h(X))=CH^1(Y\times X).\]

 In the decomposition
\[\sM(h(Y),h(X))=\sM(h(Y),h_0(X))\oplus \sM(h(Y),h_1(X))\oplus \sM(h(Y),h_2(X))\]
the first (resp. third) summand is $CH^0(Y)$ (resp. $CH^1(Y)$), in view of \eqref{eq0} and the isomorphisms $h_0(X)=\un$, $h_2(X)\simeq \L$. Comparing with Weil's formula \cite[Th. 3.9 (i)]{scholl}
\[CH^1(Y\times X) \simeq CH^1(Y)\oplus CH^1(X)\oplus \Hom(\Alb_Y,J(X))\]
and checking that the idempotents match concludes the proof.
\end{proof}

\begin{prop}[compare \protect{\cite[1.11]{scholl}, \cite[Cor. 7.8.10]{kmp} and \cite[Th. 5.2.6]{osull}}]\label{p1} Let $i\in\{ 0,1\}$. Write $w_{\le i}\sM$ for the thick subcategory of $\sM$ generated by the $h_j(X)'s$ for $j\le i$, where $X$ runs through smooth projective irreducible varieties. Then:
\begin{enumerate}
\item The inclusion functor $w_{\le i} \sM\inj \sM$ has a left adjoint, denoted by $w_{\le i}$. We have $w_{\le 0} h(X)=h_0(X)$ and $w_{\le 1} h(X)=h_0(X)\oplus h_1(X)$ for $X\in \Sm^\proj$.
\item For $M\in \sM$, the unit morphism $M\to w_{\le i} M$ is split epi, whence a split exact sequence
\[0\to w_{>i} M\to M\to w_{\le i} M\to 0\]
natural in $M$, where $w_{>i} M := \Ker(M\to w_{\le i} M)$. We have $w_{>i}h(X)=h_{>i}(X)$ for $X\in \Sm^\proj$.
\item For $X\in \Sm^\proj$, the objects $h_i(X)$ and $h_{>i}(X)$ ($i=0,1$) are independent of the choices of CK decompositions, and are functorial for the action of correspondences. 
\end{enumerate}
\end{prop}

\begin{proof} In (1), it suffices to check that the left adjoint is defined on motives of the form $h(X)$, which follows from \eqref{eq1.1} and \eqref{eq1.4}. Same reduction for (2); (3) is an obvious consequence of (1) and (2).
\end{proof}

\section{The motive of a surface (review)}\label{s3}

\subsection{The transcendental part}\label{s3.1} Let  $S$ be a surface over $k$. In \cite{kmp}, we refined Murre's Chow-K\"unneth decomposition of $h(S)\in \sM$ 
\[h(S)=\bigoplus_{i=0}^4 h_i(S)\]
by splitting the motive $h_2(S)$ into
\begin{equation}\label{eq1.6}
h_2(S)= t_2(S) \oplus\NS_S(1)
\end{equation}
where $\NS_S$ is the Artin motive associated to the N\'eron-Severi group of $S_{k_s}$. This decomposition was anticipated in \cite[end of Section 2]{ggp}. 

The motive $t_2(S)$ is called the \emph{transcendental part of $h(S)$}. It controls the Albanese kernel in the sense that
\begin{equation}\label{eq1}
T(S)\simeq \sM(\un,t_2(S)).
\end{equation}

Moreover, $t_2(S)$ is a ``birational motive'' in the following sense:
\begin{align}
\sM(M(1),t_2(S))&=0 \text{ for any } M\in \sM,\label{eq2}\\
 \sM(h(Y),t_2(S))&\simeq T(S_{k(Y)})\text{ for any (connected) } Y\in \Sm^\proj\label{eq3}\\
\intertext{\cite[Cor. 7.8.5]{kmp}. By Poincar\'e duality, we also have from \eqref{eq2}:}
\sM(t_2(S),\L)&=0.\label{eq2a}
\end{align}

 See proposition \ref{p6.1} for a higher-dimensional version of these facts.
 
 \begin{lemma}\label{l3.1} With notation as in Section \ref{s2}, the projections $h_{>1}(S)\to h_2(S)$ and $h_2(S)\to t_2(S)$ do not depend on the choice of the refined CK decomposition of $S$.
\end{lemma}

\begin{proof} For the first, it suffices to see that
\[\sM(h_4(S)\oplus h_3(S),h_2(S))=0,\]
which follows from \eqref{eq1.1} and \eqref{eq1.4} by Poincaré duality. For the second, it similarly suffices to observe that $\sM(\NS_S(1),t_2(S))=0$, which follows from \eqref{eq2}.
\end{proof}

 For simplicity, assume $Y$ geometrically connected. If it has a CK decomposition, we get from \eqref{eq3} a decomposition
 \begin{equation}\label{eq3.2}
 T(S_{k(Y)})/T(S) \simeq \bigoplus_{i=1}^{2\dim Y}\sM(h_i(Y),t_2(S))
 \end{equation}
which shows how the growth of $T(S)$ under extensions of $k$ is controlled by Chow-K\"unneth summands. A first application is if $Y$ is a curve $C$: we get thanks to \eqref{eq2}:
\begin{equation}
T(S_{k(C)})/T(S)\simeq \sM(h_1(C),t_2(S)).\label{eq0.2}
\end{equation}

Here is another application:

\subsection{Mumford's theorem} \label{s3.2}
In \eqref{eq3}, take $Y=S^n$ for some $n>0$. Then the CK decomposition of $S$ induces a CK decomposition of $S^n$. In particular, $h_2(S)^{\oplus n}$ is a direct summand of $h(S^n)$, and so is $t_2(S)^{\oplus n}$. In view of \eqref{eq3.2}, this shows that
\begin{quote} 
\it $\End_\sM(t_2(S))^n$ is a direct summand of $T(S_{k(S^n)})/T(S)$
\end{quote}
and \eqref{eq3.2} also gives the equivalences
\begin{multline*}t_2(S)\ne 0 \iff \End_\sM(t_2(S))\ne 0 \iff T(S_{k(S)})/T(S)\ne 0\\
 \iff \dim_\Q T(S_\Omega)\ge \frac{1}{2}\trdeg(\Omega/k)
\end{multline*}
if $\Omega$ is an algebraically closed field containing $k$, of transcendence degree $\ge 2$ (possibly infinite). (By \cite{ggp}, one can replace ``transcendence degree $\ge 2$'' by ``transcendence degree $\ge 1$'' by using \eqref{eq0.2},  and this is precisely what we shall refine in Part 2 in terms of the Griffiths group.)

Suppose $k$ algebraically closed. If $H$ is a classical Weil cohomology with coefficients in the field $F$, recall the orthogonal decomposition
\[H^2(S) = \NS(S)\otimes F\oplus H^2_\tr(S)\] 
for the Poincaré pairing. Since the homological realisation of $t_2(S)$ is $H^2_\tr(S)$ \cite[Prop. 7.2.3]{kmp}, a sufficient condition for $t_2(S)\ne 0$ is that $b^2>\rho$. This gives a simple motivic proof of Mumford's theorem on the ``hugeness'' of $T(S_\C)$ if $p_g>0$, in  Bloch's style \cite[App. to Lect. 1]{bloch}.

\section{Grading by coniveau}\label{s5}

The following definition and lemmas use the semi-simplicity of $\sM_\num$ (Jannsen \cite{jannsen-inv}).

\begin{defn} Let $S\in \sM_\num$ be a simple motive. We say that $S$ is \emph{primitive} if $S(-1)$ is not effective. We say that $M\in \sM_\num$ is primitive if every simple summand of $M$ is primitive.
\end{defn}

\begin{lemma}For any simple $S\in \sM_\num$, there is a unique $n\ge 0$ such that $S(-n)$ is primitive. We call $n$ the \emph{coniveau} of $S$ and write it $\nu(S)$.
\end{lemma}

\begin{proof} Write $S$ as a direct summand of $h(X)$ for some $X\in \Sm^\proj$. Suppose that $S(-n)$ is effective; then $S$ is also a direct summand of $h(X_n)(n)$ for some $X_n\in \Sm^\proj$. Therefore
\[0\ne \sM_\num(h(X_n)(n),h(X))=A^{\dim X-n}_\num(X_n\times X)\]
showing that $\nu(S)\le \dim X$.
\end{proof}

 This yields:

\begin{prop}\label{p4.2} For any $M\in \sM_\num$, there is a unique decomposition
\begin{equation}\label{eq4.2}
M=\bigoplus_{j\ge 0} M_j(j)
\end{equation}
with $M_j$ primitive. This is the \emph{coniveau decomposition} of $M$.
\end{prop}

\begin{proof} Write $M$ as a direct sum of its isotypical components and group by coniveau.
\end{proof}

Let $H$ be a classical Weil cohomology (see \S \ref{s1.1}). An object $M\in \sM_H$ has \emph{weight $i$} if $H^j(M)=0$ for $j\ne i$. In particular, $i<0$ $\Rightarrow$ $M=0$. A \emph{weight decomposition} of $M$ is a (necessarily unique) decomposition
\[M=\bigoplus_{i\ge 0} M_i\]
with $M_i$ of weight $i$. if $M=h(X)$ for some variety $X$, we recover the K\"unneth decomposition.

Let $\sM_H^*$ denote the full subcategory of $\sM_H$ consisting of motives having a weight decomposition. Then for any $M\in \sM_H^*$, the ideal
\[\Ker\left(\End_{\sM_H}(M)\to \End_{\sM_\num}(M)\right)\]
is nilpotent \cite[Prop. 5]{aknote}. Hence the decomposition \eqref{eq4.2} for $M_\num$
lifts from numerical to homological equivalence, uniquely up to isomorphism. 

\begin{lemma}\label{l2.1} Let $M\in \sM_H^*$ be of weight $i$. Then, in a lift of \eqref{eq4.2} from $\sM_\num$ to $\sM_H$, the motive $M_j$ is
\begin{itemize}
\item $0$ if $i-2j<0$;
\item an Artin motive if $i-2j=0$;
\item  of the form $h_1(A)$ for some abelian variety $A$ if $i-2j=1$.
\end{itemize}
\end{lemma}

\begin{proof} In $\sM_H^*$, $M_j$ is of weight $i-2j$, hence of the said form. Indeed, this is trivial for $i-2j\le 0$. For $i-2j=1$, writing $M_j$ as a direct summand of $h(Y)$ for some (not necessarily connected) $Y\in \Sm^\proj$, we find that $M_j$ is a direct summand of $h_1(Y)$, and conclude by \eqref{eq1.3} and Poincaré complete reducibiity.
\end{proof}

\section{The motive of certain $3$-folds}\label{s4}

Until Section \ref{s8}, we let $T$ be a geometrically connected $3$-fold admitting a CK decomposition
\begin{equation}\label{eq2.3}
h(T)=\bigoplus_{i=0}^6 h_i(T)\in \sM.
\end{equation}

Examples include $T=C\times S$,  smooth complete intersections in $\P^N$ and abelian varieties \cite[\S 3.4]{birat-pure}. 

We now ``refine'' \eqref{eq2.3} in the style of \cite[Th. 7.7.3 (iii)]{kmp}:

Applying \eqref{eq4.2} to $h_i(T)$ in $\sM_\num$, we may write it 
\begin{equation}\label{eq2.4}
h_i(T)=\bigoplus_{j\ge 0} h_{i,j}(T)(j).
\end{equation}

\begin{lemma}\label{l2.2}
The decomposition \eqref{eq2.4} lifts from numerical to rational equivalence. The motive $h_{2,1}(T)$ is an Artin motive and $h_{3,1}(T)$ is of the form $h_1(J^2)$ for some (isogeny class of) abelian variety $J^2=J^2(T)$.
\end{lemma}

\begin{proof} By Poincaré duality, we have $h_{6-i}(T)\simeq h_i(T)^*(3)$ for $i\le 3$, hence we may assume $i\le 3$. Then the only nontrivial decompositions are for $i=2,3$:
\begin{align}
h_2(T) &= t_2(T) \oplus h_{2,1}(T)(1)\label{eq2.5}\\
h_3(T) &= t_3(T) \oplus h_{3,1}(T)(1)\label{eq2.5a}
\end{align}
where we write $t_i(T):=h_{i,0}(T)$ for simplicity; the claims on $h_{2,1}(T)$ and $h_{3,1}(T)$ follow from Lemma \ref{l2.1}.

Let $p\in \sM_\num(h_i(T),h_{i,1}(T)(1))$, $q\in \sM_\num(h_{i,1}(T)(1),h_i(T))$ be the two morphisms corresponding to these decompositions. We have $pq=1$. Lift $p$ and $q$ to morphisms $\tilde p$ and $\tilde q$ in $\sM$. Since 
\[\End_\sM(h_{i,j}(T))\iso \End_{\sM_\num}(h_{i,j}(T))\]
by Lemma \ref{l2.1} and \eqref{eq1.3a}, \eqref{eq1.3}, we still have $\tilde p \tilde q=1$. Then $\tilde q\tilde p\in \End_\sM(h_i(T))$ is idempotent, hence the promised decomposition.
\end{proof}

\section{The summands of $h_*(T)$}\label{s6}

One should compare some computations in this section and the next one with those of Gorchinskiy-Guletski\v \i\   \cite[\S\S 4, 5]{gg}, although we don't use cohomology here except in the proof of Proposition \ref{p6.2}.

\subsection{The Artin summands}

The Artin motives $h_{2n,n}(T)$ are easy to compute:

\begin{lemma}\label{l2.3} For $n\in [0,3]$, $h_{2n,n}(T)$ is the Artin motive associated to the Galois module $A_n^\num(T_{k_s})$, where $k_s$ is a separable clusure of $k$.
\end{lemma}

\begin{proof} For any adequate equivalence relation, write
\begin{equation}\label{eq2.2}
A_n^\sim(T)=\sM_\sim(\L^n,h(T))=\bigoplus_{i,j}\sM_\sim(\L^n,h_{i,j}(T)(j))
\end{equation}

For $\sim=\num$, all terms with $i\ne 2n$ are $0$ for weight reasons. This leaves
\begin{multline*}
A_n^\num(T) =\sM_\num(\L^n,h_{2n}(T))\\
 = \sM_\num(\L^n,h_{2n,n}(T)(n)) = \sM_\num(\un,h_{2n,n}(T))
\end{multline*}
by definition of $h_{2n,n}(T)$ and Schur's lemma. This computation over all finite separable extensions of $k$ yields the lemma.
\end{proof}

\subsection{The abelian summand}

Here we study the abelian variety $J^2$ introduced in Lemma \ref{l2.2}. We first note:

\begin{lemma} $J^2$ contains $P_T\approx \Alb_T$ as a direct summand (up to isogeny).
\end{lemma}

\begin{proof} Let $L\in CH^1(T)$ be the class of a smooth hyperplane section. By Murre and Scholl \cite{murre,scholl}, intersection product with $L^{3-i}$ induces isomorphisms
\begin{equation}\label{eq3.1}
h_{6-i}(T)\iso h_i(T)(3-i)
\end{equation}
for $i=0,1$. For $i=1$, \eqref{eq3.1} factors through $h_3(T)(1)$, which implies that $h_1(T)$ is a direct summand of $h_{3,1}(T)$ in $\sM_\num$.
\end{proof} 

We now give two universal properties of $J^2$ in terms of numerical equivalence. For this, note that multiplication by integers $n$ on any abelian variety $A$ acts on $h_i(A)$ by the homothety of scale $n^i$ \cite{dm}. Thus the inclusion $h_1(J^2)(1)\inj h_3(T)$ (resp. the projection $h_3(T)\surj h_1(J^2)(1)$) obtained from \eqref{eq2.5a} defines a canonical element
\[P^2\in A^2_\num(J^2\times T)_\Q^{[1]} \quad \text{(resp. }P_2\in A_2^\num(T\times J^2)_\Q^{[1]})\]
(Poincar\'e classes), where the superscript $^{[1]}$ means the eigensubspace where multiplication by $m$ on $J^2$ acts by the homothety of scale $m$. 

\begin{lemma}\label{l5.1} a) The classes $P^2$ and $P_2$ verify
\[\langle P^2,P_2\rangle =1\]
for the intersection pairing.\\
b) For any abelian variety $A$ and any $\alpha\in A^2_\num(A\times T)_\Q^{[1]}$, there is a unique (isogeny class of) morphism $u:A\to J^2$ 
such that $u^*P^2=\alpha$.\\ 
c) For any abelian variety $A$ and any $\alpha\in A_2^\num(T\times A)_\Q^{[1]}$, there is a unique (isogeny class of) morphism $u:J^2\to A$ 
such that $u_*P_2=\alpha$.
\end{lemma}

\begin{proof} a) follows from the fact that the composition
\[h_1(J^2)(1)\by{P^2} h_3(T)\by{P_2} h_1(J^2)(1)\]
is the identity.
Let us prove b) (the proof of c) is dual). Such an $\alpha$ corresponds to a morphism $ h_1(A)(1)\to h(T)$ in $\sM_\num$, which factors uniquely through $P^2$ by uniqueness of the refined K\"unneth decomposition.
\end{proof}

Next, we want to compare $J^2$ with Murre's algebraic representative $\ab^2$ \cite{murreSitges}; by \cite[Th. 4.4 and Rk. 4.5]{vialetal}, it is defined over $k$ provided $k$ is perfect.  Since the sequel will only be used in Remark \ref{r7.11}, we  only sketch some arguments. 

From the split epimorphism $h(T)\surj h_1(J^2)(1)$ of Lemma \ref{l2.2}, we get a split epimorphism in $\Ab\otimes \Q$
\[
\sM(h(X)(1),h(T))\\
\to \sM(h(X)(1),h_1(J^2)(1))
\]
for any $X\in \Sm^\proj$. Using   \eqref{eq1.4a} and translating to $\Ab$, this amounts (after choosing an abelian variety $\tilde J^2$ in the isogeny class $J^2$) to an honest homomorphism $\phi(X):CH^2(X\times T)\to \tilde J^2(X)$, natural  for the action of correspondences, and a natural homomorphism $\sigma(X)$ in the opposite direction such that $\phi(X)\sigma(X)=m$ for an integer $m>0$ independent of $X$. In particular, when $k$ is algebraically closed we get an epimorphism (for $X=\Spec k$):
\[\phi(k):CH^2(T)\to \tilde J^2(k).\]

\enlargethispage*{20pt}

The naturality of $\phi$ then shows that, for any $Y\in CH^2(X\times T)$, the composition
\[X(k)\by{\theta} CH_0(X)\by{Y_*} CH^2(T)\by{\phi(k)} \tilde J^2(k)\]
(where $\theta(x)=[x]$) is induced by $\phi(X)(Y)$. 
Hence $\phi$ is regular, and so is its restriction to $CH^2(T)_\alg$. Since $\sM(\L,h_1(J^2)(1))$ consists of algebraically trivial cycles \eqref{eq1.5}, $\sigma(\Spec k)$ lands into $CH^2(T)_\alg$, so that $\phi_{|CH^2(T)_\alg}$ is still surjective. Coming back to the general case, this and the universal property of $\ab^2$ yield an epimorphism  
\[v:\ab^2_{\bar k}\Surj \tilde J^2_{\bar k}\]
where $\bar k$ is an algebraic closure of $k$, that we assume perfect.

\enlargethispage*{20pt}

\begin{prop}\label{p6.2}  
The morphism $v$ is an isogeny defined over $k$.
\end{prop}

\begin{proof}[Sketch]\footnote{The sketch consists of neglecting to keep track of some denominators in \S\S 2-3.} Let $A=\ab^2$. By \cite[Cor. 1.6.3]{murreSitges} and  \cite[Lemma 4.9]{vialetal}, there exists a correspondence $Y \in CH^2(A \times T)$ and a positive integer $r$ such that we have $\psi\circ Y_*\circ \theta_0 = r1_{A(\bar k)}$, where $\psi:CH^2(T_{\bar k})_\alg\to A(\bar k)$ is the universal regular homomorphism and $\theta_0(a) = [a]-[0]$.

We claim that we can choose $Y$ such that its image in $A^2_\num(A \times T)_\Q$ lies in $A^2_\num(A \times T)_\Q^{[1]}$. Granting this claim, the universal property of Lemma \ref{l5.1} b) provides a $k$-homomorphism $u :A\to \tilde J^2$ 
such that the map $A(\bar k)\by{r} A(\bar k)$ factors through $u (\bar k)$. But the construction of the natural transformation $\phi$ then shows that 
$u(\bar k)=rv(\bar k)$; thus $u_{\bar k}=rv$, 
and $v$ is an isogeny defined over $k$.

To prove the claim, we first note that since $A$ has a CK decomposition \cite{dm}, we may write any $Z\in CH^2(A\times T)=\sM(h(A)(1),h(T))$ as $Z=\sum Z^i$ where $Z^i\in \sM(h_i(A)(1),h(T))$. It therefore suffices to show that, if $Z^1=0$, then $\psi \circ Z_*\circ \theta_0=0$. (Indeed, we then have  $\psi \circ Y_*\circ \theta_0=\psi \circ Y^1_*\circ \theta_0$.)

For this, we use Bloch's maps $\lambda^i_X:CH^i(X)\{l\} \to H^{2i-1}_\et(X,\Q_l/\Z_l(i))$ from \cite{blochroitman}, with $l\ne \car k$: by loc. cit., Prop. 3.5, they are compatible with the action of correspondences. (This could be seen more easily by interpreting the $\lambda^i$ in a modern way in terms of étale motivic cohomology via \cite[Prop. 4.17]{cycletale}.) Let $d=\dim A$.   
By hypothesis, $Z:H^{2d-1}_\et(A,\Q_l(d))\to H^{3}_\et(T,\Q_l(2))$ is $0$ as it acts via $Z^1$, hence $Z:H^{2d-1}_\et(A,\Q_l/\Z_l(d))\allowbreak \to H^{3}_\et(T,\Q_l/\Z_l(2))$ has finite image. 
Since $\lambda^2_T$ is injective \cite[Prop. 9.2]{murreSitges},  we find that $\psi \circ Z_*\circ \theta_0$ has finite image in $A(k)\{l\}$. Since $\psi$ is regular, $\psi \circ Z_*\circ \theta_0$ is induced by an endomorphism of $A$, which must be trivial.
\end{proof}

\begin{rk}\label{r6.1}  
The cycle $Y$ appearing in the proof of Proposition \ref{p6.2} is close to Voisin's ``universal'' codimension $2$ cycle (\cite[Def. 1.5]{voisin}, see also \cite[Déf. 5.3]{ct}). The difference is that, in loc. cit., the integer $r$ is required to be equal to $1$.
\end{rk}

\section{$t_3(T)$ and the Griffiths group}\label{s7}

In this section, we do two related things:

\begin{itemize}
\item Study the uniqueness of the lifts from Lemma  \ref{l2.2}: see Corollary \ref{c6.1} and Lemma \ref{l7.1}. (Recall that the corresponding lift for a surface is unique \cite{kmp}.)
\item Relate the Griffiths group of $T$ to the motive $t_3(T)$: see Theorem \ref{t7.1} and Conjecture \ref{co1}.
\end{itemize}

\subsection{A Chow-Lefschetz condition}
To go further, we need:

\begin{hyp}\label{h1} There exists a smooth hyperplane section $\Sigma\subset T$ such that \eqref{eq3.1} is an isomorphism also for $i=2$. 
\end{hyp}

This is true in the same cases as in \S \ref{s4}:

\begin{prop}\label{p7.2} Hypothesis \ref{h1} holds if $T$ is an abelian variety, a complete intersection in $\P^N$ or for $T=C\times S$ if $\Sigma$ is defined as in \cite[Prop. 1.4.6 (ii)]{kleiman-dix}. 
\end{prop}

\begin{proof} The first case follows from  \cite{kunnemann}, the second one is trivial and the last one is proven
in Theorem \ref{tB.1}.  
\end{proof}

\subsection{Uniqueness of \eqref{eq2.5}}

\begin{prop}\label{l2.4} Let $\Sigma\by{i} T$ be a smooth hyperplane section as in Hypothesis \ref{h1}. Then the morphism $i_*:h(\Sigma)\to h(T)$ induces a split surjection $t_2(\Sigma)\Surj t_2(T)$.
\end{prop}

\begin{proof} The isomorphism \eqref{eq3.1} factors as
\[h_4(T)\by{i^*} h_2(\Sigma)(2)\by{i_*} h_2(T)(2),\]
showing that $i_*:h_2(\Sigma)\to h_2(T)$ is split surjective. Let $\lambda:h_2(T)\to h_2(\Sigma)$ be a right inverse. 
Using the decompositions \eqref{eq1.6} and \eqref{eq2.5}, write $i_*$ (resp. $\lambda$) as a $2\times 2$ matrix
\[i_*=\begin{pmatrix} a& b\\ c& d\end{pmatrix}, \quad \lambda=\begin{pmatrix} a'& b'\\ c'& d'\end{pmatrix}.\]

Then $b'=c=0$ by \eqref{eq2} and \eqref{eq2a}, and $b,c'$ are numerically equivalent to $0$ by Schur's lemma, hence $c'b=0$ by \eqref{eq1.3a}. From $i_*\lambda=1$, we get $aa'=1-bc'$ with $(bc')^2=bc'bc'=0$; thus $aa'(1+bc')=1$ and $a:t_2(\Sigma)\to t_2(T)$ has a right inverse.  
\end{proof}

\begin{cor}\label{c6.1} Under Hypothesis \ref{h1}, the lift of \eqref{eq2.5} given by Lemma \ref{l2.2} is unique, up to unique isomorphism.
\end{cor}

\begin{proof} Proposition \ref{l2.4}, \eqref{eq2} and \eqref{eq2a} imply that there are no nonzero homomorphisms between the two summands.
\end{proof}

\subsection{Uniqueness of \eqref{eq2.5a}}

\begin{lemma}\label{l7.1} Assume that $h_3(T)$ is finite-dimensional in the sense of Kimura \cite{kimura}, e.g. $T$ is of abelian type. Then the  lift of \eqref{eq2.5a} given by Lemma \ref{l2.2} is unique, up to possibly non-unique isomorphism.
\end{lemma}

\begin{proof} This follows from Kimura's nilpotence theorem \cite[Prop. 7.5]{kimura}.
\end{proof}

\subsection{Birational properties of $t_3(T)$}\label{s7.4}
The next proposition uses Voevodsky's triangulated category of motives $\DM_\gm^\eff$ \cite{voetri} (with $\Q$ coefficients). Recall that there is a canonical fully faithful functor \cite[Cor. 6.7.3]{be-vo}
\begin{equation}\label{eq6.0}
\Phi:\sM\to \DM_\gm^\eff.
\end{equation}

\begin{prop}\label{p6.1} We have $\uHom(\Z(2),\Phi(t_3(T)))=0$ in $\DM_\gm^\eff$. In particular,
\begin{equation}\label{eq6.3}
\sM(M(2),t_3(T))=0
\end{equation}
for any $M\in \sM$ and
\begin{equation}\label{eq6.4}
\sM(k)(h(Y)(1),t_3(T))\simeq \sM(k(Y))(\L,t_3(T)_{k(Y)})
\end{equation}
for any connected $Y\in \Sm^\proj$.
\end{prop}

\begin{proof} We proceed as in the proof of \cite[Th. 7.8.4 b)]{kmp}. So we reduce to showing that 
\[\DM_\gm^\eff(K)(\Z(2)[4],\Phi(t_3(T_K)[i]))=0\]
for any function field $K/k$ and any $i\in\Z$. By Poincaré duality, this group is a direct summand (in $\Ab\otimes \Q$) of
\[\DM_\gm^\eff(K)(\Z(2)[4],M(T_K)[i])\simeq H^{i+2}_\Nis(T_K,\Z(1))\simeq H^{i+1}_\Nis(T_K,\G_m).\]

Recall that this group is $0$ for $i\ne 0,-1$. For $i=0$, it corresponds to the summand  $(3,0)$ in the decomposition
\[\Pic(T_K)=CH_2(T_K)=\bigoplus_{(i,j)} \sM(K)(\L^2,h_{i,j}(T)(j)).\]

We find $\NS(T_K)$ and $\Pic^0(T_K)$ respectively as the summands $(4,2)$ and $(5,0)$ (the latter because $h_5(T)\simeq h_1(T)(2)$). Thus all other summands are $0$. For $i=-1$ we reason similarly.

This and the full faithfulness of $\Phi$ immediately implies \eqref{eq6.3}. To prove \eqref{eq6.4}, we note that \eqref{eq6.3} implies that the functor $\sM\ni M\mapsto \sM(M(1),t_3(T))$ factors through the category $\sM^\o$ of birational Chow motives of \cite{birat-pure}; we can then use the adjunction of \cite[Th. 6.6]{adjoints}.
\end{proof}

\subsection{The Griffiths group}

\begin{thm}\label{t7.1} Under Hypothesis \ref{h1}, we have canonical isomorphisms
\begin{equation}\label{eq7.1}
\Griff(T)
\simeq \sM_\alg(\L,t_3(T))\simeq \sM_\alg(t_3(T),\L^2)
\end{equation}
for any lift of \eqref{eq2.5a} to $\sM_\alg$.
\end{thm}

\begin{proof} In \eqref{eq2.2}, take $n=1$. We first examine which terms vanish for any $\sim$. This is the case of the summands for $i=0,1$. For $i=2$, we have
\[\sM_\sim(\L,h_2(T))=\sM_\sim(\L,t_2(T))\oplus \sM_\sim(\L,h_{2,1}(T)(1))= A_1^\num(T)  \]
by \eqref{eq2} and Proposition \ref{l2.4}. 

For $i=3$, we have
\[\sM_\sim(\L,h_{3,1}(T)(1)) = \sM_\sim(\un,h_1(J^2)).\]

For $i=4$, we have
\begin{multline*}
\sM_\sim(\L,h_4(T))\iso \sM_\sim(\L,h_2(T)(1))\\=\sM_\sim(\un,t_2(T)\oplus A_1^\num(T)(1))
=\sM_\sim(\un,t_2(T));
\end{multline*}
for $\sim=\rat$, the last group is a direct summand of $T(\Sigma)$ as in Proposition \ref{l2.4}. 

Finally, for $i=5,6$ we find
\begin{align*}
\sM_\sim(\L,h_5(T))\simeq \sM_\sim(\un,h_1(T)(1))&=0 \\ 
\sM_\sim(\L,h_6(T))\simeq \sM_\sim(\L,\L^3)&=0.
\end{align*}

Thus, for $\sim=\alg$, the only nonzero terms are for $(i,j)=(2,1), (3,0)$; this completes the proof of Theorem \ref{t7.1}, except for the second isomorphism. To prove it, we note that Poincar\'e duality and the uniqueness of the coniveau decomposition \eqref{eq4.2} yield an isomorphism in $\sM_\num$:
\[t_3(T)\simeq t_3(T)^*(3).\]

Since the first isomorphism of \eqref{eq7.1} is valid for any lift of $t_3(T)$ to $\sM_\alg$, we get the second one by replacing $t_3(T)$ with $t_3(T)^*(3)$. (This argument avoids the uniqueness issue from Lemma \ref{l7.1}.)
\end{proof}

\subsection{A conjecture} 

\begin{conj}\label{co1} The map $\sM(\L,t_3(T))\to \sM_\alg(\L,t_3(T))$ is bijective.
\end{conj}

\begin{prop}\label{p7.1} Conjecture \ref{co1} follows from the Bloch-Beilinson--Murre (BBM) conjectures (see \cite{jannsen}).
\end{prop}

To prove this proposition, we need a lemma:

\begin{lemma}[Weil-Bloch trick] \label{lwb}Let $F:\sM\to \Ab\otimes\Q$ be a contravariant additive functor. Then $F$ factors through $\sM_\alg$ if and only if $F(1_M\otimes \alpha)=0$ for any $M\in \sM$, any $k$-curve $\Gamma$ and any $\alpha\in \Pic^0(\Gamma)=\sM(\un,h_1(\Gamma))$. We may restrict to $M=h(Y)$ for $Y\in \Sm^\proj$.\qed
\end{lemma}

Apply this lemma to $F(M)=\sM(M(1),t_3(T))$. We shall use \cite[Prop. 5.8]{jannsen}, which Jannsen shows to be a consequence of the BBM conjectures that we now assume. If $M=h(Y)$, then for any $k$-curve $\Gamma$
\begin{multline*}
\sM(k)(h(Y)\otimes h_1(\Gamma)(1),t_3(T))\simeq\sM(k(Y))(h_1(\Gamma)(1),t_3(T))\\
\iso \sM_\num(k(Y))(h_1(\Gamma)(1),t_3(T))\osi \sM_\num(k)(h_1(\Gamma)(1),t_3(T))=0.
\end{multline*}

Here the first isomorphism is \eqref{eq6.4}, the second one is a special case of \cite[Prop. 5.8]{jannsen}, the third one is given by the full faithfulness of $\sM_\num(k)\to \sM_\num(k(Y))$ \cite[Prop. 5.5]{adjoints}, and the last vanishing is by Schur's lemma and by definition of $t_3(T)$. Hence $F$ verifies the condition of Lemma \ref{lwb}.\qed

\begin{rk}\label{r7.11} Conjecture \ref{co1} is striking, as it predicts  in view of \eqref{eq7.1} that $\Griff(T)$ is (up to isogeny) cut off $CH^2(T)$ by an idempotent self-correspondence of $T$, namely the one defining $t_3(T)$. According to Beilinson's vision, it also predicts an isomorphism
\[\Griff(T)\simeq \Ext_{\sM\sM}^1(\L,t_3(T)) \]
where $\sM\sM$ is the hypothetical abelian category of mixed motives, see \cite[(2.3)]{jannsen}.

It is also related to a conjecture of Nori \cite[p. 227]{jannsen2}: codimension $2$ cycles in the kernel of the Abel-Jacobi map are algebraically equivalent to $0$. Jannsen proves in loc. cit., Th. 6.1 that this holds under the BBM conjectures and the standard conjecture B.  

By the proof of Theorem \ref{t7.1}, the only summands of $CH^2(T)$ which contain algebraically trivial cycles are \allowbreak $\sM(\L,h_{3,1}(T)(1))=J^2(k)$, \allowbreak $\sM(\L,h_4(T))=\sM(\un,t_2(T))$ and $\sM(\L,t_3(T))$. The first two are fully algebraically trivial \eqref{eq1.5}. On the other hand,  the Abel-Jacobi map is injective on $J^2(k)$ by   
Proposition \ref{p6.2} and \cite[Th. 1.11.1]{murreSitges}. Thus Conjecture \ref{co1} refines Nori's conjecture as follows, in the case of a $3$-fold $T$ verifying Hypothesis \ref{l2.4}: 

\begin{quote} \emph{The kernel of the Abel-Jacobi map on $CH^2(T)$ equals $\sM_\sim(\L,h_4(T))=\sM_\sim(\un,t_2(T))$.} 
\end{quote}

By Proposition \ref{p7.1} and \cite[Th. 6.1]{jannsen2}, this statement follows from the BBM conjectures.

In higher dimensions, nothing new is expected to happen: one easily sees that the BBM conjectures plus the standard conjecture B imply that $CH^2(X)_\hom \to CH^2(T)_\hom$ is split injective in $\Ab\otimes \Q$ if $T$ is a $3$-dimensional general successive hyperplane section of a smooth projective variety $X$. 
\end{rk}

\section{The case $T=C\times S$}\label{s8}

We now assume that $T=C\times S$, where $C$ is a curve and $S$ is a surface; since $T$ is geometrically connected, so are $C$ and $S$. 

\subsection{A reformulation of Theorem \ref{t7.1}}

\begin{prop}\label{t1} 
There are canonical isomorphisms
\begin{equation}\label{eq2.1}
\sM_\alg(h_1(C),t_2(S))\simeq \sM_\alg(\L,t_3(C\times S))\simeq \Griff(C\times S).
\end{equation}
\end{prop}

\begin{proof} In view of Proposition \ref{p7.2}, the second isomorphism is Theorem \ref{t7.1}. Let us prove the first.  By duality, we may write
\[\sM_\alg(h_1(C),t_2(S))=\sM_\alg(\L,h_1(C)\otimes t_2(S)).\]

The K\"unneth isomorphism in $\sM$
\[h_3(C\times S)\simeq \bigoplus_{i=0}^2 h_i(C)\otimes h_{3-i}(S)\]
gets converted via \eqref{eq1.6}, \eqref{eq2.5} and \eqref{eq3.1} into an isomorphism
\begin{multline}\label{eq9.3}t_3(C\times S)\oplus h_1(J^2)(1)\\
\simeq  h_1(S)(1)\oplus h_1(C)\otimes \NS_S(1) \oplus h_1(C)\otimes t_2(S) \oplus h_1(S)(1).
\end{multline}

Working modulo numerical equivalence, we may write 
\begin{equation}\label{eq9.2}
h_1(C)\otimes t_2(S) \simeq t'_3(C\times S)\oplus M(1)
\end{equation}
by Proposition \ref{p4.2}; proceeding as in Lemma \ref{l2.1}, we see that $M=h_1(B)$ for some abelian variety $B$. 

Proceeding as in the proof of Lemma \ref{l2.2}, we may lift \eqref{eq9.2} to a decomposition in $\sM$.  Since $\sM_\alg(\un,h_1(A))=0$ for any abelian variety $A$ by \eqref{eq1.5}, \eqref{eq9.3} and \eqref{eq9.2} together yield \eqref{eq2.1}. \end{proof}

\begin{cor} There is a natural action of $\End J(C)^\op\times \End t_2(S)$ on $\Griff(C\times S)$.\qed
\end{cor}

\begin{cor}\label{c8.1} We take the same notation as in the proof of Proposition \ref{t1}.\\ 
a) There is an isogeny
\[J^2\approx B\times (\Alb_S)^2\times \Alb_C\otimes \NS_S.\]
b) If $t_2(S)$ is finite dimensional (equivalently, if $h(S)$ is finite-dim\-ension\-al), the Chow motives $t_3(C\times S)$ and $t'_3(C\times S)$ are isomorphic.
\end{cor}

\begin{proof} Consider \eqref{eq9.3} modulo numerical equivalence, and insert the value of $h_1(C)\otimes t_2(S) $ given by \eqref{eq9.2}. Taking isotypic components, we get two isomorphisms in $\sM_\num$:
\begin{gather}
t_3(C\times S) \simeq t'_3(C\times S)\label{eq9.1}\\
h_1(J^2)\simeq h_1(B)\oplus 2h_1(S)\oplus h_1(C)\otimes \NS_S. \label{eq9.1a}
\end{gather}

The isogeny of a) follows from \eqref{eq9.1a}. If $t_2(S)$ is finite dimensional, so are all terms in \eqref{eq9.3} and \eqref{eq9.2}. In particular, the Chow motives $t_3(C\times S)$ and $t'_3(C\times S)$ are finite-dimensional, and b) follows from \eqref{eq9.1} and Kimura's nilpotence theorem  \cite[Prop. 7.5]{kimura}.
\end{proof}

\begin{rks}\label{p9.1} a) Without assuming the finite-dimensionality of $t_2(S)$, one gets an isomorphism in $\sM$
\[h_3(C\times S)\simeq t_3(C\times S)\oplus h_1(J^2)(1)\simeq t'_3(C\times S)\oplus h_1(J^2)(1)\]
but it is not clear whether one can  cancel the summand  $h_1(J^2)(1)$ unconditionally.

b) Assuming Conjecture \ref{co1}, the isomorphisms \eqref{eq9.2} and \eqref{eq0.2} yield an isomorphism in $\Ab\otimes \Q$
\[T(S_{k(C)})/T(S)\simeq B(k)\oplus \Griff(C\times S).\]

Since $\Griff(C\times S)$ is invariant under algebraically closed extensions, this suggests that $E\mapsto T(S_{E(C)})/T(S_E)$ is representable by a $k$-group scheme whose identity  component is an abelian variety. Can one prove this a priori?
\end{rks}

\subsection{Curves mapping to surfaces}\label{s8.2} Let $\psi\in \sZ_1(C\times S)$ be a correspondence. It induces a morphism in $\sM$
\[\psi_*:h(C)\to h(S).\]

Choose a CK decomposition $(\pi_i^C)$ of $C$ and a refined CK decomposition $(\pi_i^S)$ of $S$, with $\pi_2^S=\pi_2^\alg(S)+\pi_2^\tr(S)$: we get a composite morphism
\begin{equation}\label{eq9.5}
\psi_\#:h_1(C)\Inj h(C)\by{\psi_*} h(S)\Surj t_2(S)
\end{equation}
hence, by Proposition \ref{t1}, a class $[\psi_\#]$ in $\Griff(C\times S)$. 
This class depends on the choice of the CK decompositions; note however that, thanks to \eqref{eq1.1}, it does not depend on the choice of the $0$-cycles on $C$ and $S$ used to define $\pi_0^C$ and $\pi_0^S$. In Lemma \ref{l9.1} b) below, we show that a slight variant of $\psi_\#$ does not depend on any choice. \emph{The main purpose of Part \ref{P2} will be to study the non-vanishing of $[\psi_\#]$, in special circumstances.}

We want to compare \eqref{eq9.5} with a canonical morphism defined below \eqref{eq8.1}.

Thanks to Proposition \ref{p1} (2), $\psi_*$ induces a commutative diagram of split exact sequences in $\sM$, independent of the choice of the CK decompositions (see Section \ref{s2} for the notation):
\begin{equation}\label{eq8.2}
\begin{CD}
0@>>> h_{>1}(C)@>i_C>> h_{>0}(C)@>p_C>> h_1(C)@>>> 0\\
&&@V{\psi_{>1}}VV @V{\psi_{>0}}VV@V{\psi_1}VV\\
0@>>> h_{>1}(S)@>i_S>> h_{>0}(S)@>p_S>> h_1(S)@>>> 0.
\end{CD}
\end{equation}

Let $\psi^*:h(S)\to h(C)(1)$ be the transpose of $\psi$. The composition 
\begin{equation}\label{eq8.4}
\L=h_2(C)=h_{>1}(C)\by{\psi_{>1}} h_{>1}(S)\by{\psi^*} h(C)(1)\to h_0(C)(1)=\L
\end{equation}
is multiplication by the self-intersection number $[C]^2$, where $[C]=\psi_*1\in \Pic(S)$ with $1$ the canonical generator of $CH^0(C)$. It is  an isomorphism if $[C]^2\ne 0$, which we now assume. (This hypothesis is verified if $[C]$ is ample.) Then we can define $h_{>1}(S)/h_{>1}(C)=\Coker \psi_{>1}$. On the other hand, by \eqref{eq1.3} we can also define 
\begin{equation}\label{eq8.5}
h_1(\psi)=\Ker\psi_1
\end{equation}
which does not depend on the choice of the CK decompositions. The snake lemma then yields a morphism $h_1(\psi)\to h_{>1}(S)/h_2(C)$, that we compose with the projection $h_{>1}(S)/h_2(C)\to h_2(S)/h_2(C)$ to obtain
\begin{equation}\label{eq8.1}
\tilde\psi:h_1(\psi)\to h_2(S)/h_2(C).
\end{equation}

In view of \eqref{eq3}, the projection $h_2(S)\to t_2(S)$ factors through a morphism $\rho:h_2(S)/h_2(C)\to t_2(S)$ and

\begin{lemma}\label{l9.1} Assume $[C]^2\ne 0$ as above. Then\\
a) The diagram in $\sM$
\[\begin{CD}
h_1(\psi)@>\tilde \psi>> h_2(S)/h_2(C)\\
@V\iota VV @V\rho VV\\
h_1(C)@>\psi_\#>> t_2(S)
\end{CD}\]
commutes. We write $\hat \psi:=\psi_\#\iota = \rho \tilde \psi$.\\
b) The morphisms $\tilde \psi$ and $\hat \psi$ do not depend on the choices of the (refined) CK decompositions of $h(C)$ and $h(S)$.\\
c) In $\sM_\alg$, $\hat\psi =0$  $\iff$ $\tilde \psi=0$.
\end{lemma}

\begin{proof} a) We can use Diagram \eqref{eq8.2} to compute $\psi_\#$ and $\tilde\psi$ by using the section of $p_C$ and the retraction of $i_S$ given by $\pi_1^C$ and $\sum_{i>1} \pi_i^S$. 

b) follows from a) and Lemma \ref{l3.1}.

c) If $\rho\tilde\psi =0$ in $\sM_\alg$, then $\tilde \psi$ factors in $\sM_\alg$ through $h_2^\alg(S)/h_2(C)$. But $\sM_\alg(h_1(\psi),h_2^\alg(S)/h_2(C))=0$ by \eqref{eq1.5} and Poincaré duality.
\end{proof}

Given its importance in the sequel, we give a name to the motive $h_1(\psi)\otimes t_2(S)$:

\begin{nota}\label{n8.1}
$M(\psi)=h_1(\psi)\otimes t_2(S)$.
\end{nota}

Thus $M(\psi)$ is a direct summand of $h_1(C)\otimes t_2(S)$, and we have a similar decomposition to \eqref{eq9.2}
\begin{equation}\label{eq9.2a}
M(\psi) \simeq t_3(\psi)\oplus h_1(B_\psi)(1)
\end{equation}
where $t_3(\psi)$ (resp. $B_\psi$) is a direct summand of $t'_3(C\times S)$ (resp. of the abelian variety $B$ of Corollary \ref{c8.1} a), up to isogeny). The proof is the same as for Lemma \ref{l2.2}.

\begin{defn}\label{d8.1} We call $B_\psi$ the \emph{exceptional summand} of $B$ (or of $J^2(C\times S)$).
\end{defn}

\begin{rk} If $\psi_1$ is an epimorphism (which happens when $C$ is a smooth hyperplane section of $S$), the complementary summand of $B_\psi$, $B'$, is such that $h_1(B')(1)$ is a direct summand of $h_1(S)\otimes t_2(S)$, and the largest of this type modulo numerical equivalence; so $B'$ is independent of $\psi$ and (in principle) computable purely in terms of $S$. This justifies the terminology ``exceptional''.
\end{rk}

\subsection{Variation over a base}\label{s8.3} Let $X\in \Sm(k)$. Recall De\-nin\-ger-Murre's category $\sM(X)$ of relative Chow motives, constructed on \allowbreak smooth projective $X$-schemes \cite{dm}: it is $2$-contravariant for morphisms in $\Sm(k)$. We shall need the following generalisation of the previous picture ``over $X$'': 

\begin{itemize}
\item $S$ is a $k$-surface, to which we associate the constant relative $X$-surface $S_X:=S\times X\to X$;
\item $C\to X$ is a relative curve, and $\psi\in CH^2(C\times_X S_X)=CH^2(C\times_k S)$ is a relative Chow correspondence.
\end{itemize}

Then $C$ has a CK decomposition, defined as usual, and any refined CK decomposition of $h(S)$ pulls back to a refined CK decomposition of $S_X$. The modified class $[\psi]$ of \eqref{eq9.5} then makes sense in $\sM(X)$, as well as the composition \eqref{eq8.4}, whose bijectivity can be checked at any point of $X$. Moreover,

\begin{lemma}\label{l8.2} Suppose $X$ connected, with generic point $\eta=\Spec K\by{j} X$. Then\\
a) The functor $j^*:\sM(X)\to \sM(K)$ is full.\\
b) If $C,C'$ are two relative curves over $X$, the map
\[\sM(X)(h_1(C),h_1(C'))\to \sM(K)(h_1(C_\eta),h_1(C'_\eta))\]
is bijective, and the thick subcategory of $\sM(X)$ generated by the $h_1(C)$'s is abelian semi-simple.
\end{lemma}

\begin{proof} a) follows immediately from the surjectivity of the map $CH^*(Y)\allowbreak\to CH^*(Y_\eta)$ for any $Y\in \Sm^\proj(X)$. If $Z\subset X$ is a closed subset purely of codimension $c$, we have a more precise exact sequence
\[CH^{i-c}(Y_{|Z})\to CH^i(Y)\to CH^i(Y_{|X-Z})\to 0.\]

For $c>1$, this implies an isomorphism $\sM(X)(h(C),h(C'))\iso \sM(X-Z)(h(C_{|X-Z}),h(C'_{|X-Z}))$, hence a fortiori the same isomorphism when replacing $h$ by $h_1$. For $c=1$, this isomorphism for the $h_1$'s
 follows by a diagram chase from the formula
\[\sM(X)(h_1(C),h_1(C'))\simeq \Coker(CH^1(C)\oplus CH^1(C')\to CH^1(C\times_X C'))\]
and the same one over $X-Z$. The first claim of b) follows from these two cases by passing to the limit, and the second claim follows from this and \eqref{eq1.3}.
\end{proof}

By Lemma \ref{l8.2} b),  the motive $h_1(\psi)$ of \eqref{eq8.5} makes sense in $\sM(X)$ (note that we may write $h_1(S)$ as a direct summand of $h_1(D)$ for some smooth hyperplane section $D$ of $S$), as well as $\hat{\psi}$ and $M(\psi)$  (Lemma \ref{l9.1} and Notation \ref{n8.1}). So does \eqref{eq8.1} when \eqref{eq8.4} is bijective; then Lemma \ref{l9.1} a) holds in $\sM(X)$.

\subsection{The coniveau filtration}\label{s9.1}

Let $H$ be a Weil cohomology. For any $X\in \Sm^\proj$ and any $i\ge 0$, we have the coniveau filtration on $H^i(X)$ \cite[(2.0.6)]{katz}:
\begin{equation}\label{eq4.1}
N^jH^i(X) = \sum_{(T,z)} \IM\left(H^{i-2j}(T)(-j)\by{\cl(z)^*} H^i(X)\right)
\end{equation}
where $T\in \Sm^\proj$ and $z\in A_H^{\dim T + j}(X\times T)$. This generalises to any $M\in \sM$ by the formula
\begin{equation}\label{eq4.1a}
N^jH^i(M) = \sum_{(N,z)} \IM\left(H^{i-2j}(N)(-j)\by{\cl(z)^*} H^i(M)\right)
\end{equation}
where $N\in \sM$ and $z\in \sM(M,N(j))$: indeed, for $M=h(X)$, the images of \eqref{eq4.1} and \eqref{eq4.1a} coincide as one sees by writing $N$ in \eqref{eq4.1a} as a direct summand of some $h(T)$. (This formula makes it clear that $N^*H(M)$ is a descending filtration on $H(M)$.) Replacing the indexing sets $\sM(M,N(j))$ by $\sM(M,N(j))\otimes F$, where $F$ is the field of coefficients of $H$, does not change the coniveau filtration. One may also use cycles modulo homological equivalence.

We note:

\begin{lemma}\label{l8.1} Write $H^3(\psi):=H^3(M(\psi))$. Suppose that $N^1H^3(\psi)=0$. Then $B_\psi=0$, where $B_\psi$ is as in Definition \ref{d8.1} (see \eqref{eq9.2a}).
\end{lemma}

\begin{proof} By definition of the coniveau filtration, we have $H(h_1(B_\psi)(1))\subseteq N^1H^3(\psi)$. Therefore, the hypothesis implies that $h_1(B_\psi)(1)$ vanishes in $\sM_H$. Hence so does $h_1(B_\psi)$, which implies that $B_\psi=0$ by \eqref{eq1.3}.
\end{proof}
\newpage

\part{Proofs of Theorems \ref{T2} and \ref{T3}}\label{P2}

\section{The $l$-adic Abel-Jacobi map}\label{s9}

\subsection{The $l$-adic realisation}\label{s9.2}

Let $S$ be a scheme, and let $l$ be a prime number invertible on $S$. We write $\Sm_c(S)$ for the category  of smooth separated compactifiable morphisms  $X\to S$ with morphisms proper $S$-morphisms, and $S_c(S,\Q_l)$ for the category of constructible $\Q_l$-adic sheaves over $S$ \cite[1.4.2]{jouanolou}, $D^b_c(S,\Q_l)$ for the ``derived category of constructible $\Q_l$-adic sheaves'' (\cite[1.1.3]{weilII}, \cite{ekedahl})  and
\[R_{l,c}:\Sm_c(S)\to D^b_c(S,\Q_l)\]
for the (contravariant) functor which associates to $X\by{f} S\in \Sm_c(S)$ the object $Rf_!\Q_l$. If $V$ is an $S$-scheme, we write
\[\tilde D^b_c(V/S,\Q_l)= 2\text{-}\colim_\sV D^b_c(\sV,\Q_l)\]
where $V\to \sV$ runs through the $S$-morphisms with $\sV$ an $S$-scheme of finite type (see \cite[Def. 6.3]{SGA4.VI} for $2\text{-}\colim$). The functor $R_{l,c}$ factors through a functor
\[\tilde R_{l,c}:\Sm_c(V)\to \tilde D^b_c(V/S,\Q_l).\]

If $S=\Spec \Z[1/l]$, we drop it from the notation. 

With the definition of \cite{ekedahl}, the category $\tilde D^b_c(V/S,\Q_l)$ is triangulated and has a canonical $t$-structure with heart $\tilde S_c(V/S,\Q_l)\allowbreak = 2\text{-}\colim_\sV S_c(\sV,\Q_l)$. For $X\in \Sm_c(V)$, we write $ \tilde H^i_{l,c}(X/S)\allowbreak =H^i(\tilde R_{l,c}(X))$; note that $\tilde H^i_{l,c}(X/S)\allowbreak =\tilde H^i_{l}(X/S)$ when $X$ is proper (ordinary $l$-adic cohomology): in this case, this $S$-sheaf is locally constant by the smooth and proper base change theorem. Note the forgetful functor
\begin{equation}\label{eq9.11}
\omega_V^l:\tilde S_c(V,\Q_l)\to S_c(V,\Q_l)
\end{equation}
which forgets the ``arithmetic monodromy action'' to only remember the ``geometric'' one.

Suppose that $V=\Spec k$. The functor $(\tilde R_{l,c})_{| \Sm^\proj(k)}$ extends to an additive (contravariant) functor
\[\tilde R_l:\sM(k)\to \tilde D^b_c(k/S,\Q_l)\]
such that
\[\tilde R_{l}(h_i(X)) = \tilde H^i_l(X/S)[-i]\]
for $X\in \Sm^\proj(k)$ and $i\in\Z$, for any CK decomposition of $X$. In particular, $\tilde R_l(\L)=\Q_l(-1)[-2]$.

\begin{prop}\label{p9.3} Let $M\in \sM$. Suppose given a splitting $\tilde R_l(M)\simeq \bigoplus_i \tilde H^i_l(M)[-i]$. Then, for any $n\ge 2$, the composition
\begin{multline*}
\sM(M,\L^n)\by{\tilde R_l} \Hom(\tilde R_l(\L^n),\tilde R_l(M))\\ \to \Hom(\tilde R_l(\L^n),\tilde H^{2n-1}_l(M)[-2n+1])\simeq \Ext^1_{\tilde S_c(k/S,\Q_l)}(\Q_l(-n),\tilde H^{2n-1}_l(M))\\
\to \Ext^1_{\tilde S_c(k/S,\Q_l)}(\Q_l(-n),\tilde H^{2n-1}_l(M)/N^{n-1}\tilde H^{2n-1}_l(M))
\end{multline*} 
factors through $\sM_\alg(M,\L^n)$. \end{prop}

\begin{proof} We may assume $k$ algebraically closed. Let $\alpha\in \sM(M,\L^n))$ be such that $\alpha\mapsto 0\in \sM_\alg(M,\L^n)$. Then there is a curve $D/k$ and classes $\beta\in J(D)(k)=\sM(h_1(D),\L)$, $\gamma\in \sM(M,h_1(D)(n-1))$ such that $\alpha$ factors as
\[M\by{\gamma} h_1(D)(n-1)\by{\beta}\L^n.\]

By definition of the coniveau filtration, the composition
\[H^1_l(D)(n-1)\to H^{2n-1}_l(M)\to H^3_l(M)/N^{n-1}H^{2n-1}_l(M)\]
is $0$, hence the conclusion.
\end{proof}

As a special case, we get an Abel-Jacobi homomorphism
\begin{equation}\label{eq9.10}
\Griff^2(X)\otimes \Q_l\to \Ext^1_{\tilde S_c(k/S,\Q_l)}(\Q_l(-1),\tilde H^3_l(X)/N^1\tilde H^3_l(X))
\end{equation}
for any smooth projective $X$. One can check that it is canonical: this 
was developed more integrally in the preliminary version of \cite{cycletale} (see \cite[App. B]{cycletale0}).

\begin{rk} For $k=\C$, the same statement as Proposition \ref{p9.3} is valid when replacing the $l$-adic realisation by the Hodge realisation with values in the derived category of mixed Hodge structures, with the same proof.
\end{rk}

\subsection{A cohomological condition of nonvanishing} We come back to the situation of  \S \ref{s8.2}, and assume that  $[C]^2\ne 0$.

\begin{prop}\label{p9.4} Take the notation of Lemma \ref{l9.1} a), Notation \ref{n8.1} and Lemma \ref{l8.1}. 
Suppose that  $N^1H^3_l(\psi)\allowbreak =0$. Then:   
\begin{multline*}
\tilde R_l(\hat \psi)\ne 0 \Rightarrow\hat \psi\ne 0\text{ in } \sM_\alg(h_1(\psi),t_2(S))\\
\Rightarrow \psi_\#\ne 0 \text{ in } \sM_\alg(h_1(C),t_2(S))\simeq \Griff(C\times S).
\end{multline*} 
\end{prop}

Recall that the hypothesis also implies that $B_\psi=0$ (Lemma \ref{l8.1}).

\begin{proof} The second implication is trivial. For the first, we use the rigidity of the category $\sM[\L^{-1}]$ and the fact that $\tilde R_l$ is a symmetric monoidal functor. This yields a commutative diagram
\[\begin{CD}
\sM(h_1(\psi),t_2(S)) @>\sim>> \sM(M(\psi),\L^2)\\
@V{\tilde R_l}VV @V{\tilde R_l}VV\\
\Ext^1_{\tilde S_c(k,\Q_l)}(\tilde H^2_{l,\tr}(S),\tilde H^1_l(h_1(\psi)))@>\sim>> \Ext^1_{\tilde S_c(k,\Q_l)}(\Q_l(-2), \tilde H^3_l(\psi))
\end{CD}\]
and we get the conclusion by taking $n=2$ in Proposition \ref{p9.3}.
\end{proof}

\subsection{Interpretation of $\tilde \psi$ as an extension}\label{s9.3} We remain in the situation of  \S \ref{s8.2}, but assume further that  $\psi$ is an embedding $C\inj S$. For simplicity, we identify $C$ to its image in $S$. Let $V=S-C$; the exact triangle for  cohomology with proper supports
\[\tilde R_{l,c}(V)\to \tilde R_{l,c}(S)\by{\psi^*} \tilde R_{l,c}(C)\by{+1}\]
yields the corresponding long exact sequence
\[\tilde H^1_l(S)\by{\psi^*} \tilde H^1_l(C)\to \tilde H^2_{l,c}(V) \to \tilde H^2_l(S)\by{\psi^*} \tilde H^2_l(C)\]
whence a short exact sequence
\begin{equation}\label{eq8.3}
0\to \tilde H^1_l(\psi)\to \tilde H^2_{l,c}(V)\to \tilde H^2_{l,\prim}(S)\to 0
\end{equation}
where $\tilde H^1_l(\psi)=\tilde H^*_l(h_1(\psi))$ and $\tilde H^2_{l,\prim}(S)=\tilde H^*_{l}(h_2(S)/h_2(C))$. This yields an extension class
\[\sE\in \Ext^1_{\tilde S_c(k,\Q_l)}(\tilde H^2_{l,\prim}(S),\tilde H^1_l(\psi)).\]

On the other hand, the morphism $\tilde\psi$ of \eqref{eq8.1} gives a morphism
\[\tilde R_{l,c}(\tilde\psi):\tilde H^2_{l,\prim}(S)[-2]\to \tilde H^1_l(\psi))[-1]\] 
which yields another extension class $\sE'\in \Ext^1_{\tilde S_c(k,\Q_l)}(\tilde H^2_{l,\prim}(S),\tilde H^1_l(\psi))$. The following proposition will not be used in the sequel, but is good to know:

\begin{prop}\label{l9.3} The classes $\sE$ and $\sE'$ coincide.
\end{prop}

\begin{proof} This follows from applying the functor $\tilde R_{l,c}$ to Diagram \eqref{eq8.2}.
\end{proof}

\section{Variations of Hodge structures}\label{s10}

\subsection{The Hodge realisation}\label{s.hodge}  
Let $X$ be a smooth connected complex algebraic curve. For any subring $A$ of $\R$, we have the category $\sV(X,A)$ of good variations of mixed $A$-Hodge structures over $X$ (\cite[\S 4]{andrecomp}, see p. 8 for ``good''). We shall write
\begin{equation}\label{eq10.Z}
\sV(X)=\IM(\sV(X,\Z)\to \sV(X,\Q))
\end{equation}
(variations of mixed Hodge structure of integral origin). This is a full tannakian subcategory of the (neutral) tannakian category $\sV(X,\Q)$ \cite[Appendix]{sz}.

Let $\Loc(X)$ be the tannakian category of local systems of finite-dimensional $\Q$-vector spaces over $X$ (for the analytic topology).  Let also $x\in X(\C)$ be a rational point. We have a naturally commutative diagram of exact faithful $\otimes$-functors
\begin{equation}\label{eq9.8}
\begin{CD}
\MHS=\sV(\Spec \C)@>\pi_\sV^*>> \sV(X)@>x_\sV^*>> \MHS\\
@V\omega VV @V{\omega_X} VV@V\omega VV \\
\Vec_\Q=\Loc(\Spec \C)@>\pi^*>> \Loc(X)@>x^*>> \Vec_\Q
\end{CD}
\end{equation}
where $\pi:X\to \Spec \C$ is the structural morphism, $\MHS$ is the category of graded-polarisable mixed $\Q$-Hodge structures, $\omega,\omega_X$ forget the Hodge structures and $x^*\pi^*=\Id_{\Vec_\Q}$, $x_\sV^*\pi_\sV^*=\Id_{\MHS}$. 

The fibre functors $x^*$ and $\omega\circ x_\sV^*$ define tannakian groups $\Pi_1(X,x)$ and $T_x$, which fit in a  commutative diagram of affine $\Q$-groups
\begin{equation}\label{eq9.9}
\begin{CD}
\MT@<\pi_{\sV}<< T_x@<x_{\sV}<< \MT\\
@A\omega^* AA @A{\omega_X^*} AA@A\omega^* AA \\
1@<\pi<< \Pi_x@<x<< 1
\end{CD}
\end{equation}
where $\MT$ is the Tannakian group of $\MHS$ (the universal Mumford-Tate group).  In particular, $\pi_{\sV}$ is split surjective. 
Note that $\Pi_x$ is the algebraic envelope of the topological fundamental group $\pi_1(X,x)$ (for its action on objects of $\Loc(X)$).

\begin{thm}\label{t1V}  $\omega_X^*(\Pi_x)$ is normal in $T_x$, and we have an exact sequence of affine $\Q$-groups
\begin{equation}\label{eq1V}
\Pi_x\by{\omega_X^*} T_x\by{\pi_{\sV}} \MT\to 1
\end{equation}
where 
$\pi_{\sV}$ has the section $x_{\sV}$. 
\end{thm}

For the proof, see Appendix \ref{yves} (Theorem \ref{t2A}). Actually, we will use this theorem only in the case of pure variations of geometric origin, so that the difficult theorem of the fixed part of Schmid-Steenbrink-Zucker used in the appendix (Theorem \ref{t1A}) can be replaced with the much simpler theorem of the fixed part of Deligne \cite[4.1.1]{HII}. (I am indebted to Yves Andr\'e for this clarification.)

Let $y\in X(\C)$ be another point; any path from $x$ to $y$ yields an isomorphism of fibre functors $x^*\iso y^*$; hence the collection of $\Pi_x$ and $T_x$ assemble to $X$-groupoids\footnote{= functors from the fundamental groupoid $\pi_1(X)$ to the category of affine group schemes.} $\Pi,T$, and \eqref{eq1V} gets promoted to an exact sequence of $X$-groupoids
\begin{equation}\label{eq2V}
\Pi\by{\omega_X^*} T\by{\pi_\sV} \MT\to 1
\end{equation}
where $\MT$ is viewed as a constant $X$-groupoid. The sections $x_{\sV*}$ \emph{do not} define a splitting of $\pi_\sV$ as a morphism of $X$-groupoids. The sequel of this subsection explains that this is ``almost true'' on integral objects of $\sV(X)$.

For $V\in \sV(X)$, write $T_x(V)$ for the image of $T_x$ in $\GL(\omega (V_x))$, and $\Pi_x(V),\MT_x(V)$ for the images of $\omega_X^*(\Pi_x)$ and $x_{\sV}(\MT)$: with the notation of \cite[Lemma 4]{andrecomp}, we have $\MT_x(V)=G_x$ and $\Pi_x(V)^0 = H_x$ where $^0$ means neutral component. 
We shall also need

\begin{defn}\label{d10.1} An object $V\in \sV(X)$ is \emph{connected} if $\Pi_x(V)$ is connected for some (hence all) $x\in X(\C)$.
\end{defn}

\begin{cor}\label{tandre} a) The collections $(T_x(V))_{x\in X}$ and $(\Pi_x(V))_{x\in X}$ define local systems on $X$. For any $x\in X(\C)$, $\Pi_x(V)$ is normal in $T_x(V)$ and $\Pi_x(V)^0$ is normal in $T_x(V)^0$; we have 
\begin{align}
\Pi_x(V)\MT_x(V) &=T_x(V)\label{eq10.4}\\
\Pi_x(V)^0\MT_x(V)&=T_x(V)^0\label{eq10.5}\\
\pi_0(\Pi_x(V))&\to \pi_0(T_x(V)) \text{ is surjective.}\label{eq10.6}
\end{align}
b) For $x\in X(\C)$, the following are equivalent:
\begin{thlist}
\item $\Pi_x(V)^0\subseteq  \MT_x(V)$. 
\item $\MT_x(V)=T_x(V)^0$.
\end{thlist}
If $V$ is simple and connected (Definition \ref{d10.1}), then $V_x$ is simple for such $x$. \\
c) The set 
\[\Exc(V)=\{x\in X(\C)\mid \text{(i) and (ii) do not hold}\}\] 
is \emph{countable} and its complement is path-connected; in particular, $x\mapsto \MT_x(V)$ defines a local system on $X(\C)\setminus \Exc(V)$. 
\end{cor}

\begin{proof} In a), the first statements and \eqref{eq10.4} follow from Theorem \ref{t1V} and \eqref{eq2V}; \eqref{eq10.5} then follows from the connectedness of $\MT(V_x)$; finally, \eqref{eq10.6} means that $\Pi_x(V)T_x(V)^0=T_x(V)$, which follows from \eqref{eq10.4} and \eqref{eq10.5}. This makes the equivalence in b) obvious, and the last assertion then follows from \eqref{eq10.6} and (ii). Finally, the proof of the second assertion of  \cite[Lemma 4]{andrecomp} shows that $\Exc(V)$ is a union of countably many proper closed analytic subvarieties of $X(\C)$, hence is countable since $\dim_\C X(\C)=1$. For the path-connectedness, see \cite[Lemma p. 276]{mustafin}.
\end{proof}

\begin{rks} a) In particular, $\Pi_x(V)^0$ is normal in $\MT_x(V)$ under the assumptions of b). This is proven 
in \cite[Th. 1]{andrecomp}, which also shows that $\Pi_x(V)^0\subseteq  D(\MT_x(V))$ (the derived subgroup). I learnt \eqref{eq10.6}, the implication (i) $\Rightarrow$ (ii) and their consequence on the simplicity of $V_x$ from Y. Andr\'e, whose proofs were different.\\
b) In fact, \eqref{eq10.6} is bijective: see Corollary \ref{C.15}. We will not need this refinement.
\end{rks}

\subsection{Hodge coniveau} Let $V\in \sV(X)$. We say that $V$ is \emph{effective} if, for any $i$, the bundle $(\gr_i^WV)^{p,q}$ is $\ne 0$ only for $p\ge 0$ and $q\ge 0$. We write $\nu(V)=\sup\{n\in \Z\mid V(n) \text{ is effective}\}$. Whence the \emph{Hodge coniveau filtration}:
\[N^j_H V= \sum_{\begin{smallmatrix}W\subseteq V\\ \nu(W)\ge j\end{smallmatrix}} W\subseteq V.  \]

Let $\sV(X)^\eff$ be the full subcategory $\{V\in \sV(X)\mid V\text{ is effective}\}$. Then $\nu(V)\ge j$ if and only if $V(j)\in \sV(X)^\eff$. Since $\sV(X)^\eff$ is closed under quotients, $N^j_H$ is right adjoint to the inclusion functor $\sV(X)^\eff(j)\inj \sV(X)$, which implies the identity
\begin{equation}\label{eq9.7}
N^j_H (V_1\oplus V_2)=N^j_H V_1\oplus N^j_H V_2.
\end{equation}

One should beware that $N^*_H$ does \emph{not} commute with specialisation: if $V\in \sV(X)$, then the obvious inclusion, for $x\in X(\C)$,
\[(N^j_HV)_x\subseteq N^j_H(V_x)\]
need not be an equality in general. However:

\begin{prop}\label{p10.1} For any $V\in \sV(X)$, let
\[\Exc^N(V)=\{x\in X\mid N^1_HV_x\ne 0\}.\]
Suppose that $V$ is semi-simple, connected (Definition \ref{d10.1}) and that $N^1_H V=0$. Then
\[\Exc^N(V)\subseteq \bigcup_S \Exc(S)\]
where $\Exc$ is as in Corollary \ref{tandre} c) and $S$ runs through the simple constituents of $V$.
\end{prop}

\begin{proof} By \eqref{eq9.7} we have $\Exc^N(V_1\oplus V_2)= \Exc^N(V_1)\cup \Exc^N(V_2)$; since every direct summand of $V$ is connected, we may assume $V$ simple. Since $\dim V_x^{p,q}=\rank V^{p,q}$ for any $x$, we have $N^1_H V_x\ne V_x$ for any $x$. But $V_x$ is simple for $x\notin \Exc(V)$ by Corollary \ref{tandre} b),  hence the conclusion.
\end{proof}

The following easy lemma is a key step in the proof of Theorem \ref{T2}: it is a special feature of Hodge theory.

\begin{lemma}\label{l9.2} Let $V_1,V_2\in \sV(X)$. Then $\nu(V_1\otimes V_2)=\nu(V_1)+\nu(V_2)$. If $V_1,V_2\in \sV(X)^\eff$ and $V_1\otimes V_2$ is simple, then ($N^1_H V_1=0$ $\bigwedge$ $N^1_H V_2=0$) $\Rightarrow$ $N^1_H(V_1\otimes V_2)=0$. 
\end{lemma}

\begin{proof} The inequality $\nu(V_1\otimes V_2)\ge \nu(V_1)+\nu(V_2)$ is obvious. By twisting, it suffices to show that $\nu(V_1)=0$ and $\nu(V_2)=0$ implies $\nu(V_1\otimes V_2)=0$.  This is clear, since $\gr_0^W(V_1)^{0,q}\ne 0$ for some $q$ and $ \gr_0^W(V_2)^{0,r}\ne 0$ for some $r$ implies $\gr_0^W(V_1\otimes V_2)^{0,q+r}\ne 0$. This readily implies the second claim.
\end{proof}

For the next proposition, we write $\PHS\subset \MHS$ for the full subcategory of pure (polarisable) Hodge structures: it is semi-simple.

\begin{prop}\label{p9.2} Let $V\in \sV(X)$. Assume that $\omega_X(V)$ is absolutely irreducible, in the sense that its endomorphism ring is $\Q$. Let  $W\in \PHS$. Then:\\
a) If $W$ is simple, $V\otimes \pi_\sV^*W$ is simple.\\
b) Suppose $V$ and $W$ effective. If $N^1_H V=0$ and $N^1_H W=0$, then $N^1_H(V\otimes \pi^*W)=0$. Moreover, $N^1_H(V_x\otimes W)\ne V_x\otimes W$ for any $x\in X$.
\end{prop}

\begin{proof} a) (I am grateful to \url{http://mathoverflow.net/questions/208731} for its help in this proof.) 
We use Theorem \ref{t1V}. Take $x\in X$. We may view $V$ as a representation of $T_x$ and $W$ as a representation of $\MT$. The restriction of $V\otimes \pi_\sV^* W$ to $\Pi_x$ is semi-simple as a direct sum of copies of $V$, and its subrepresentations $E$ are in $1-1$ correspondence with the sub-vector spaces of $W$ by $E\mapsto (V^*\otimes E)^{\Pi_X}\subseteq \End_{\Pi_x}(V)\otimes W=W$ (using the absolute irreducibility of $V_{|H}$). Moreover, $E$ is a $T_x$-submodule of $V\otimes W$ if and only if $(V^*\otimes E)^{\Pi_x}$ is an $\MT$-submodule of $W$.

b) This follows from a) (by reducing to $W$ simple) and Lemma \ref{l9.2}.
\end{proof}

\begin{cor}\label{c9.1} Consider the hypotheses of Proposition \ref{p9.2} b). Suppose moreover $V
$  connected. Then 
\[\Exc^N(V\otimes \pi_\sV^*W)\subseteq \bigcup_i \Exc(V\otimes \pi_\sV^*W_i)\]
where $W_i$ runs through the irreducible components of $W$, $\Exc^N$ is as in Proposition \ref{p10.1} and $\Exc$ is as in Corollary \ref{tandre} c). In particular, $\Exc^N(V\otimes \pi_\sV^*W)$ is countable.
\end{cor}

\begin{proof}   This follows from Proposition \ref{p9.2}, Lemma \ref{l9.2} and Proposition \ref{p10.1}.
\end{proof}

\subsection{Comparison isomorphisms and coniveau filtrations}\label{s10.3} Let us go back to the category $\sM(X)$ of relative Chow motives considered in \S \ref{s8.3}. The functors
\[\Sm^\proj(X) \ni (\sX\by{f}X)\mapsto  R^if_*\Q\in \sV(X)\]
induce a graded $\otimes$-functor
\[\H_B^*:\sM(X)\to \sV^*(X).\]

There is a similar enriched realisation functor
\[\H_l^*:\sM(X)\to \tilde S_c^*(X,\Q_l)\]
to the graded category of ``arithmetic'' $l$-adic sheaves $\tilde S_c^*(X,\Q_l)$ of \S \ref{s9.2}. Composing with the forgetful functors $\omega_X$ of \eqref{eq9.8} and $\omega_X^l$ from \eqref{eq9.11}, we get relative Weil cohomologies
\[H_B^*:\sM(X)\to \Loc^*(X), \quad H_l^*:\sM(X)\to \Loc^*(X,\Q_l)\]
where $\Loc(X,\Q_l)$ is the category of local systems of $\Q_l$-vector spaces. 
The comparison isomorphisms $R^*f_*\Q_l\simeq (R^*f_*\Q)\otimes \Q_l$ of \cite[V, (3.5.1)]{arcata} yield an isomorphism of functors
\begin{equation}\label{eq10.3}
H_l^*\simeq H_B^*\otimes_\Q \Q_l.
\end{equation}

For $M\in \sM(X)$, we may define the coniveau filtration $N^*$ on $H^*_B(M)$ and $H^*_l(M)$ by Formula \eqref{eq4.1a}. Clearly
\begin{equation}\label{eq10.8}
N^rH^j_B(M)\subseteq N^r_H H^j_B(M):=\omega_X(N^r_H\H^j_B(X))
\end{equation}
(we call the right hand side the \emph{Hodge coniveau filtration}.)

We have the trivial lemma:

\begin{lemma}\label{l10.1} The isomorphism \eqref{eq10.3} respects the coniveau filtrations.\qed
\end{lemma}

Let now $K/k$ be any regular extension, whence a functor $\sM(k)\ni M\mapsto M_K\in  \sM(K)$. If $l$ is a prime number $\ne \car k$ and $M\in \sM$, we have an isomorphism
\begin{equation}\label{eq10.7}
H_l(M)\iso H_l(M_K)
\end{equation}
by invariance of $l$-adic cohomology under separably closed base change. 

\begin{prop}[cf. \protect{\cite[Prop. 2.2]{katz}}]\label{p9.5} The isomorphism \eqref{eq10.7} respects the coniveau filtrations.
\end{prop}

\begin{proof}  Let $X\in \Sm^\proj(k)$. By definition of the coniveau filtration, the isomorphism $H^i_l(X)\iso H^i_l(X_K)$ yields an inclusion $N^jH^i_l(X)\subseteq N^jH^i_l(X_K)$ for all $j\ge 0$. Let $x\in H^i_l(X)$ be such that $x_K\in N^jH^i_l(X_K)$. There is $Y\in \Sm^\proj(K)$, $y\in H_l^{i-2j}(Y)(-j)$ and a correspondence $\alpha\in CH_{\dim X -j}(X_K\allowbreak\times_K Y)$ such that $x_K=\alpha^* y$. Spread $Y$ to a smooth projective $f:\sY\to S$ with $S\in \Sm(k)$, and lift $\alpha$ to $\tilde \alpha\in CH_{\dim X -j+\dim S}(X_S\allowbreak\times_S \sY)$. For a closed point $s\in S$, we have $y_s=sp_s(y)\in H^{i-2j}_l(\sY_s)$ and $\alpha_s=y^*(\tilde\alpha)\in CH_{\dim X -j}(X\times_k \sY_s)$, where $\sY_s=f^{-1}(s)$; obviously, $x_{k(y)}=\alpha_s^*y_s$.
\end{proof}

We may now state

\begin{prop} \label{p10.2} Let $X$ be a smooth connected algebraic $\C$-curve and $M\in \sM(X)$. Suppose that $N^j_H \H^i_B(M_U)=0$ for some $(i,j)$ and any nonempty Zariski open subset $U\subseteq X$, where $M_U$ is the base change of $M$ to $U$\footnote{One can show that the case $U=X$ implies the general case.}. Let $M_\eta\in \sM(K)$ be the ``generic fibre'' of $M$, where $K=\C(X)$.  Then we have $N^j H^i_l(M_\eta)=0$ for any prime number $l$.
\end{prop}

\begin{proof} We have to show that any morphism $\phi:M_\eta\to N(j)$ in $\sM(K)$ induces $0$ on $l$-adic cohomology. We may find $U$ such that $\phi$ comes from a morphism $\phi_U:M_U\to \sN(j)$, with $\sN\in \sM(U)$. It suffices to show that $\phi_U$ induces $0$ on $l$-adic cohomology, and by Lemma \ref{l10.1} we may replace the latter by Betti cohomology. This now follows from the inclusion \eqref{eq10.8} and the hypothesis.
\end{proof}

\section{Lefschetz pencils}\label{s.lef}

Our aim in this section is to prove the generic case of Theorem \ref{T2} (see Corollary \ref{c11.1}). 

\subsection{Preparations} Let $S\in \Sm^\proj(k)$ be a geometrically connected surface. Consider  a Lefschetz pencil \cite[(5.6)]{weilI}:
\begin{equation}\label{eq9.4}
\vcenter{
\xymatrix{
C_\eta \ar[r]\ar[d] &\tilde S \ar[r]^\pi\ar[d]_{f}& S\\
\eta \ar[r]& \P^1
}}
\end{equation}
where $\tilde S$ is the incidence variety and $\eta$ is the generic point of $\P^1$.

We have the closed immersion 
\begin{equation}\label{eq11.3}
\psi:\tilde S\inj S\times \P^1.
\end{equation}

Let $U\subset \P^1$ be the smooth locus of $f$; the restriction $\psi_U:\tilde S_U\inj S\times U$ of $\psi$ to $U$
yields a relatively ample smooth divisor, and the formalism of \S \ref{s8.3} applies.
 Thus we get motives
\begin{equation}\label{eq11.1}
h_1(\psi_U), M(\psi_U)=h_1(\psi_U)\otimes t_2(S\times U/U)\in \sM(U)
\end{equation}
and a morphism
\begin{equation}\label{eq11.2}
\widehat{\psi_U}:h_1(\psi_U)\to t_2(S\times U/U)
\end{equation}
(see Lemma \ref{l9.1}).

Let $l$ be a prime number. In $H^1_l(\tilde S_U)$, we have two sub-local systems:
\begin{itemize}
\item $H^1_l(\psi_U):=H^1_l(h_1(\psi_U))$;
\item $\underline{E}$, the local system of vanishing cycles \cite[(6.1)]{weilI}.
\end{itemize} 

We assume that $\underline{E}\ne 0$: it is possible up to composing the given projective embedding of $S$ with a Veronese embedding, cf. \cite[Cor. 6.4]{katz1}. This forces $U\ne \P^1$, hence $U$ is affine. Then:

\begin{prop}\label{h10.1} a) The two local systems $H^1_l(\psi_U)$ and $\underline{E}$ are canonically isomorphic.\\
b) For $u\in U$, the algebraic envelope of the monodromy action on $H^1_l(\psi_U)_u=H^1_l(\psi_u)$ is $\Sp(H^1_l(\psi_u))$, where $\Sp$ is with respect to the symplectic pairing induced by the isomorphism of a).\\
c) If $k\subseteq \C$, the same picture holds when replacing $H_l$ by $H_B$. In particular, $H^1_B(\psi)\in \sV(\P^1)$ (cf. \eqref{eq10.Z}).
\end{prop}

\begin{proof} a) Write $f_U$ for the restriction of $f$ to $U$ and $g_U$ for the projection $S\times U \to U$. By definition, we have an exact sequence
\[R^1(g_U)_* \Q_l\to R^1(f_U)_* \Q_l\to H^1_l(\psi_U)\to 0. \]

For $u\in U$, let $C_u=f^{-1}(u)$ and $E_u\subset H^1_l(C_u)$ be the subspace of vanishing cycles \cite[(5.2)]{weilI}, i.e. the fibre of $\underline{E}$ at $u$. The Picard-Lefschetz formula \cite[(5.8) a) 3)]{weilI} shows that $E_u^\perp= H^1_l(C_u)^{\pi_1^\et(U,u)}$, and the Hard Lefschetz theorem (which, for a surface, follows algebraically from \eqref{eq3.1} for $i=1$) implies that $E_u$ is nondegenerate for the Poincaré duality pairing. Therefore, it suffices to show that  $H^1_l(C_u)^{\pi_1^\et(U,u)}=\IM(H^1_l(S_{k(u)})\to H^1_l(C_u))$ for any $u\in U$; since these are local systems, it suffices to show this at one point $u$, that we choose to be the generic point $\eta$. 

Since $U$ is affine, $U_{\bar k}$ has cohomological dimension $1$ and the Leray spectral sequence for $f_U$ yields an epimorphism
\[H^1_l(\tilde S)\Surj H^0_l(U,R^1(f_U)_*\Q_l)=H^1_l(C_\eta)^{\pi_1^\et(U,\eta)}.\]

But $\pi:\tilde S\to S$ is a blow-up, hence $H^1_l(S)\iso H^1_l(\tilde S)$ and $H^1_l(S)\to H^1_l(C_\eta)^{\pi_1^\et(U,\eta)}$ is surjective; so is a fortiori $H^1_l(S_{k(\eta)})\to H^1_l(C_\eta)^{\pi_1^\et(U,\eta)}$.

b) now follows from a) and the Kazhdan-Margulis theorem \cite[Th. (5.10)]{weilI}, and the first statement of c) follows from \eqref{eq10.3}; the second one holds because the vanishing cycles are integrally defined.
\end{proof}

More easily:

\begin{lemma}\label{l11.1} If $k\subseteq \C$, we have $N^1_HH^2_{B,\tr}(S)=0$.
\end{lemma}

\begin{proof} The Lefschetz (1,1) theorem shows that $N^1H^2_B(S)=N^1_HH^2_B(S)$, and $N^1H^2_{B,\tr}(S)=N^1H^2_{B}(S)\cap H^2_\tr(S)=0$ by definition of $H^2_\tr(S)$.
\end{proof}

\subsection{The main results}

\begin{thm} \label{t10.1} Suppose $\car k=0$. Then,\\
a) $N^1H^3_l(\psi_{U'})\allowbreak =0$ for any nonempty open subset $U'\subseteq U$,  and $N^1H^3_l(\psi_u)\allowbreak\ne H^3_l(\psi_u)$ for all $u\in U_{(0)}$.\\
b) $N^1H^3_l(\psi_\eta)=0$.\\
c) If $k\subseteq \C$, there exists a countable subset $E(\psi)\subset U(\C)$ such that,  for any $u\in U(k)\setminus E(\psi)$, one has   $N^1 H^3_l(\psi_u)=0$.\\
d) For $u=\eta$ or $u\in U(k)\setminus E(\psi)$, one has $B_{\psi_u}=0$ where  $B_{\psi_u}$ is the exceptional part of $J^2(C_u\times S_{k(u)})$ (see Definition \ref{d8.1}).
\end{thm}

\begin{proof} For a), using Lemma \ref{l10.1} and Proposition \ref{p9.5} we reduce to $k$ finitely generated over $\Q$, then to  $k=\C$, and then replace $l$-adic cohomology by Betti cohomology. It suffices to prove the statements with $N^1$ replaced by $N^1_H$. We  apply Proposition \ref{p9.2} b) by taking $V=H^1_B(\psi_{U'})$, $W=H^2_{B,\tr}(S)$: Proposition \ref{h10.1} (resp.  Lemma \ref{l11.1}) gives its hypothesis for $V$ (resp. $W$) -- note that the vanishing of $N^1_H H^1_B(\psi_{U'})$ is trivial.  b) then follows from Proposition \ref{p10.2}. Since $V$ is connected by Proposition \ref{h10.1}, c) follows from Corollary \ref{c9.1}. Finally, d) follows from the above and Lemma \ref{l8.1}.
\end{proof}

\begin{rk} I don't know if Theorem \ref{t10.1} b) holds in positive characteristic. By contrast, the next result holds over any $k$.
\end{rk}

\begin{thm}[cf. \protect{\cite[Th. 2]{ggp}}]\label{c10.1} Suppose that $b^2>\rho$. Then $\tilde R_l(\widehat{\psi_\eta})\ne 0$.
\end{thm}

The proof is given in the next subsection.

\begin{cor}\label{c11.1} Theorem \ref{T2} (i) is true.
\end{cor}

\begin{proof} This follows from Theorem \ref{t10.1} b), Theorem \ref{c10.1} and Proposition \ref{p9.4}.
\end{proof}

\begin{rk}\label{r11.1} Applying \eqref{eq3.2} with $Y=\tilde S$ and using \eqref{eq2}, we get
\begin{multline*}
T(S_{k(\tilde S)})/T(S)\simeq \sM(h_1(\tilde S),t_2(S))\oplus \sM(t_2(\tilde S),t_2(S))\\
\simeq  \sM(h_1(S),t_2(S))\oplus \End_\sM(t_2(S)).
\end{multline*}

The first isomorphism comes again from \eqref{eq2} plus the isomorphisms \eqref{eq3.1}; the second one comes from the fact that $h_1(\tilde S)\iso h_1(S)$ and $t_2(\tilde S)\iso t_2(S)$ since $\tilde S\to S$ is a birational morphism. On the other hand, $T(S)\iso T(S_{k(\P^1)})$, either directly or by \eqref{eq0.2} since $h_1(\P^1)=0$. Letting $K=k(\P^1)$, we thus get 
\begin{multline*}
\sM(K)(h_1(C_\eta),t_2(S_K))\simeq T(S_{k(\tilde S)})/T(S_K)\\
\simeq \sM(h_1(S),t_2(S))\oplus \End_\sM(t_2(S))
\end{multline*}
and, using \eqref{eq2.1}, we find a surjection in $\Ab\otimes \Q$:
\begin{equation}\label{eq11.8}
\sM(h_1(S),t_2(S))\oplus \End_\sM(t_2(S))\Surj \Griff(C_\eta\times_K S_K).
\end{equation}  

For simplicity, assume $b^1=0$: this is the case e.g. if $S$ is a K3 surface. Then $h_1(S)=0$ and the strong form of Bloch's conjecture \cite[Conj. 1.8]{bloch}:
\[\End_\sM(t_2(S))\iso \End_{\sM_\hom}(t_2(S))?\]
predicts that $\dim_\Q \Griff(C_\eta\times_K S_K)\otimes \Q<\infty$.
\end{rk}

\subsection{Proof of Theorem \ref{c10.1}} 
Without loss of generality, we may and do assume $k$ algebraically closed. 

Let $V\subseteq \P^1$ be an open subset, and let $\tilde S_V= f^{-1}(V)$. The projection $f_V:\tilde S_V\to V$ yields a trace map
\begin{equation}\label{eq11.9}
H^2_l(\tilde S_V)\by{(f_V)_*}  H^0(V)(-1).
\end{equation}

\begin{lemma}\label{l11.2} Let $j_V$ be the open immersion $\tilde S_V\inj \tilde S$, and let  $a:H^2_{l,\tr}(\tilde S)\inj H^2_{l}(\tilde S)$ be induced by $\pi_2^{S,\tr}$. Then:\\
a) The composition $H^2_{l,\tr}(\tilde S)\by{a} H^2_{l}(\tilde S)\by{j_V^*} H^2_l(\tilde S_V)$ is injective.\\
b) The composition $(f_V)_* j_V^* a$ is $0$.
\end{lemma}

\begin{proof} a) By semi-purity, $\Ker j_V^*$ is contained in the image of $\NS(\tilde S)\otimes \Q_l(-1)\by{\cl} H^2_l(\tilde S)$. The claim follows since $\IM a$ is the orthogonal complement of this image, by construction of $\pi_2^{S,\tr}$ \cite[\S 7.2]{kmp}.

b) It suffices to handle the case $V=\P^1$. Let $x\in H^2_{l,\tr}(\tilde S)$ and $y\in H^2_l(\P^1)$. By the projection formula, we have
\[\langle f_* ax,y\rangle_{\P^1} =\langle  ax,f^* y\rangle_{\tilde V} = 0\]
for the same reason as in the proof of (1). Thus $f_* ax=0$.
\end{proof}

In Lemma \ref{l11.2}, suppose $V\subseteq U$. In particular $V$ is affine, hence has étale cohomological dimension $1$ and the Leray spectral sequence for the projection $f_V:\tilde S_V\to V$ yields a short exact sequence
\begin{multline}\label{eq11.4}
0\to H^1(V,R^1(f_V)_*\Q_l)\to H^2_l(\tilde S_V)\\
\by{\epsilon_V} H^0(V,R^2(f_V)_*\Q_l)\to 0
\end{multline}
where the edge homomorphism $\epsilon_V$ coincides with the trace map of \eqref{eq11.9}.  By Lemma \ref{l11.2}, and \eqref{eq11.4},  the  map $a$ then yields an injection 
\[\tilde a_V:H^2_{l,\tr}(\tilde S)\inj H^1(V,R^1(f_V)_*\Q_l).\]

Let $h_1(\psi_V)$ denote the restriction of $h_1(\psi_U)$ to $\sM(V)$  (see \eqref{eq11.1}), and let $H^1_l(\psi_V)=H_l(h_1(\psi_V))\in S_c(V,\Q_l)$ be its realisation.

\begin{lemma}\label{l11.3} Let $j_V$ be the inclusion $V\inj \P^1$. The image of $\tilde a_V$ is contained in $H^1(V, j_V^*j_*\underline{E}) \simeq H^1(V, H^1_l(\psi_V))$.
\end{lemma}

\begin{proof} The Leray spectral sequence for $f$:
\[H^p(\P^1,R^qf_*\Q_l)\Rightarrow H^{p+q}(\tilde S,\Q_l)=:H^{p+q}(\tilde S)\]
yields a $3$-step filtration $F^p H^2(\tilde S)$ on $H^2(\tilde S)$ with successive quotients $H^0(\P^1,R^2f_*\Q_l)\simeq \Q_l(-1)$, $H^1(\P^1,R^1f_*\Q_l)$ and $H^2(\P^1,f_*\Q_l)=$\break $H^2(\P^1,\Q_l)=\Q_l(-1)$. By Lemma \ref{l11.2} b), we have $\IM a\subset F^1 H^2(\tilde S)$.  

Let $j:U\inj \P^1$ be the inclusion. Then $R^1f_*\Q_l \iso j_*j^*R^1f_*\Q_l$ \cite[(5.8)]{weilI} and, by the proof of Proposition \ref{h10.1} a), the orthogonal complement of $j_*\underline{E}$ in $R^1f_*\Q_l$ is constant, hence
\[H^1(\P^1,j_*\underline{E})\iso  H^1(\P^1,R^1f_*\Q_l).\]

Therefore, the image of $\tilde a_V$ is contained in $H^1(V,j_V^*j_*\underline{E})$. But we have $j_V^*j_*\underline{E}\simeq H^1_l(\psi_V)$ by Proposition \ref{h10.1} a).
\end{proof}

Note that the closed immersion $\psi$ of \eqref{eq11.3} factors as
\[\tilde S\by{\gamma} \tilde S\times \P^1\by{\Pi} S\times \P^1\]
where $\gamma$ is the graph of $f$ and $\Pi=\pi\times 1_{\P^1}$. Restricting to the open subset $V\subseteq U$, we get an induced factorisation
\[\tilde S_V\by{\gamma_V} \tilde S\times V\by{\Pi_V} S\times V\]
and $\gamma_V$ induces a morphism $\widehat{\gamma_V}:h_1(\gamma_V)\to t_2(\tilde S\times V/V)=t_2(\tilde S)_V$ analogous to \eqref{eq11.2}, whence a morphism in $D^b_c(V,\Q_l)$
\begin{equation}\label{eq11.5}
H^2_{l,\tr}(\tilde S)_V[-2]\to H^1_l(\gamma_V)[-1]
\end{equation}
with $H^1_l(\gamma_V)=H_l(h_1(\gamma_V))$, as usual. Applying $\Hom(\Q_l[-2],-)$ to \eqref{eq11.5}, we get a homomorphism
\begin{equation}\label{eq11.6}
H^2_{l,\tr}(\tilde S)=H^0(V,H^2_{l,\tr}(\tilde S)_V)\to H^1(V,H^1_l(\gamma_V)).
\end{equation}

Write $\widehat{a_V}: H^2_{l,\tr}(\tilde S)\inj H^1(V, H^1_l(\psi_V))$ for the injection given by Lemma \ref{l11.3}. Then

\begin{lemma} \label{l11.4} a) The direct summands $h_1(\psi_V)$ and $h_1(\gamma_V)$ of $h_1(\tilde S_V)$ coincide.\\
b) In view of a), the homomorphism \eqref{eq11.6} is equal to $\widehat{a_V}$, up to sign.
\end{lemma}

\begin{proof} a) By definition, $h_1(\psi_V) = \Ker(h_1(\tilde S_V)\by{(\psi_V)_*} h_1(S\times V/V))$, and similarly for $h_1(\gamma_V)$. But $(\Pi_V)_*:h_1(\tilde S\times V/V)\to h_1(S\times V/V)$ is split since $\pi$ is a blow-up with smooth centre.

b) Under $R_l$, the CK decomposition of $h(\tilde S_V/V)$ yields a direct sum decomposition in $D^b_c(V,\Q_l)$
\[R(f_V)_*\Q_l\simeq \bigoplus_{q=0}^2 R^q(f_V)_* \Q_l[-q]\]
whence a splitting of the Leray spectral sequence
\[H^n(\tilde S_V,\Q_l)=D^b_c(V,\Q_l)(\Q_l,R(f_V)_*\Q_l)\simeq \bigoplus_{q=0}^2 H^{n-q}(V,R^q(f_V)_* \Q_l)\]
from which the claim follows easily.
\end{proof}

\begin{proof}[End of proof of Theorem \ref{c10.1}] We have a naturally commutative diagram of categories and functors
\[\begin{CD}
\sM(k) @>e_\sM>> \sM(V)\\
@V\tilde R_l VV @V\tilde R_l VV\\
D^b_c(k,\Q_l)@>e_D >> D^b_c(V,\Q_l)
\end{CD}\]
where $D^b_c(k,\Q_l)$ is simply the bounded derived category of $\Vec_{\Q_l}$.
The composite homomorphism
\begin{multline*}
\sM(k)(t_2(S),t_2(S))\by{e_\sM} \sM(V)(t_2(S)_V,t_2(S)_V)\\
\by{\widehat{\psi_V}^*}\sM(V)(h_1(\psi_V),t_2(S)_V)
\end{multline*}
sends tautologically $1_{t_2(S)}$ to $\widehat{\psi_V}$. Therefore, the composite homomorphism
\begin{multline}\label{eq11.7}
\End_{\Q_l}(H^2_{l,\tr}(S))=\End_{D^b_c(k,\Q_l)}(H^2_{l,\tr}(S)[-2])\\
\by{e_D} \End_{D^b_c(V,\Q_l)}(H^2_{l,\tr}(S)_V[-2])\\
\by{\widehat{R_l(\psi_V)}_*}D^b_c(V,\Q_l)(H^2_{l,\tr}(S)_V[-2],H^1_l(\psi_V)[-1])
\end{multline}
sends $ 1_{H^2_{l,\tr}(S)}$ to $R_l(\widehat{\psi_V})$. 
The hypothesis $b^2>\rho$ means that $1_{H^2_{l,\tr}(S)}\ne 0$, hence, to conclude, it suffices to show that \eqref{eq11.7} is injective for any $V\subseteq U$. Since $V$ is irreducible, $e_D$ is an isomorphism. To show the injectivity of $\widehat{R_l(\psi_V)}_*$, we may replace the left (constant) term $H^2_{l,\tr}(S)_V[-2]$ by $\Q_l[-2]$, and are left to show that the map
\begin{multline*}
D^b_c(V,\Q_l)(\Q_l[-2],H^2_{l,\tr}(S)_V[-2])\\
\by{\widehat{R_l(\psi_V)}_*}D^b_c(V,\Q_l)(\Q_l[-2],H^1_l(\psi_V)[-1]) 
\end{multline*}
is injective.

But $\pi:\tilde S\to S$ is a birational morphism, hence $t_2(\pi)$ is an isomorphism by \eqref{eq2} and \cite[Cor. 2.4.2]{birat-pure}; thus we are reduced to the injectivity of the map
\begin{multline*}
H^2_{l,\tr}(\tilde S)=D^b_c(V,\Q_l)(\Q_l[-2],H^2_{l,\tr}(\tilde S)_V[-2]) \\
\by{\widehat{R_l(\gamma_V)}_*}D^b_c(V,\Q_l)(\Q_l[-2],H^1_l(\gamma_V)[-1]) = H^1(V,H^1_l(\gamma_V))
\end{multline*}
where we used Lemma \ref{l11.4} a) to identify $H^1_l(\psi_V)$ with $H^1_l(\gamma_V)$. This map is none else than \eqref{eq11.6}. The conclusion now follows from Lemma \ref{l11.4} b).
\end{proof}

\begin{rk}This fills in details in \cite[Proof of Th. 2]{ggp}. (The reader should beware that there are typos in this proof.) The main add-on is Lemma \ref{l11.4} b).
\end{rk}

\section{Specialisation}\label{s.sp}

In this section, $k$ is a subfield of $\C$.

\subsection{Coniveau and Lefschetz pencils}\label{s12.1} We keep the notation of Section \ref{s.lef}. We shall use the following application of a theorem due independently to Serre \cite[10.6, theorem]{mw} and Terasoma \cite[Th. 2]{terasoma}:

\begin{prop}[cf. \protect{\cite[\S 2]{ggp}}]\label{p12.1} Let $k$ be a field of characteristic $\ne l$, and let 
\[0\to V\to E\to W\to 0\]
be a nonsplit extension of $l$-adic representations of $G_{k(t)}$ (the absolute Galois group of $k(t)$). We assume that $E$ is unramified away from a finite set $S$ of places of $k(t)/k$. Then the set
\[\Omega=\{x\in \A^1(k)\setminus S\mid  E_x \text{ is split}\}\]
is thin. (See \cite[9.1, Def.]{mw} for the definition of thin.)
\end{prop}

\begin{proof} Let $G=\IM(G_{k(t)}\to GL(E))$: this is an $l$-adic Lie group. By either of the two references quoted before the proposition, the set $\Omega'$ of $x\in \A^1(k)\setminus S$ such that the decomposition group $G_x$ at $x$ is not equal to $G$ is thin.
\end{proof}

\begin{rk} By a transfer argument, the condition $(G:G_x)<\infty$ is sufficient for $E_x$ not to be split.
\end{rk}

\begin{cor}\label{c12.1} The set
\[\{u\in U(k)\mid R_l(\widehat{\psi_u})=0\}\]
is thin.
\end{cor}

\begin{proof}  This follows from Theorem \ref{c10.1} and Proposition \ref{p12.1}, applied to the extension
\[0\to H^1_l(\psi_\eta)\to E\to H^2_{l,\tr}(S)\to 0 \]
 defined by $R_l(\widehat{\psi_\eta})$.
\end{proof}

\begin{lemma}\label{l12.2} Let $K_1$ be a field of characteristic $0$. Then $K=K_1(t)$ is Hilbertian (see \cite[9.5, Def.]{mw} for the definition of Hilbertian). Moreover, $\card(\A^1(K)-\Omega)=\card(K)$ for any thin subset $\Omega\subset \A^1(K)$.
\end{lemma}

\begin{proof} The first statement is \cite[9.5, Rk. 5]{mw}; the second one follows from examining its proof. Namely, $\{at+b\mid a,b\in K\}\setminus \Omega\supseteq U(K)$ for a suitable nonempty open subset $U\subseteq \A^1_K\times \A^1_K$.
\end{proof}

\begin{thm}\label{t12.6} Suppose $k=\C$. Let
\[\Exc^\sM(\psi)=\{u\in U(\C)\mid (\psi_u)_\#= 0\in \Griff(C_u\times S)\otimes \Q\}.\]
Then $U(\C)\setminus\Exc^\sM(\psi)$ is uncountable.
\end{thm}

\begin{proof} There is an uncountable subfield $K\subset \C$, of the form $K_1(t)$, such that our Lefschetz pencil is defined over $K$. Indeed, the Lefschetz pencil is defined over some finitely generated subfield $K_0\subset \C$; pick a transcendence basis $(t_\alpha)_{\alpha\in A}$ of $\C/K_0$, and let $K_1=K_0((t_\alpha)_{\alpha\in A-\{\beta\}})$ where $\beta\in A$ is chosen, and $K=K_0(t_\alpha)$. The claim now follows from Theorem \ref{t10.1} c), Corollary \ref{c12.1}, Lemma \ref{l12.2} and Proposition \ref{p9.4}.
\end{proof}

\subsection{Andr\'e's motives}\label{s12.2} Proposition \ref{p12.1} is sufficient to get Corollary \ref{c12.1} over a finitely generated field $k$, but not to ensure that the set $U(k)\setminus E(\psi)$ of Theorem \ref{t10.1} c) is nonempty.  For this, we now have to pass from Mumford-Tate groups as in Theorem \ref{tandre} to motivic Galois groups in the sense of André \cite{incond}. Moreover, we have to strengthen the hypothesis on $S$. 

We write $\Mot(k)$ for the category denoted by $\sM_k$ in \emph{loc. cit.}, 4.5, Ex. (i).  This is a semi-simple $\Q$-linear Tannakian category \cite[Th. 0.4]{incond}, provided with canonical $\otimes$-functors 
\begin{equation}\label{eq12.1}
\sM(k)\to \Mot(k)\by{\H} \PHS\by{\omega} \Vec_\Q
\end{equation}
refining the Hodge realisation on $\sM(k)$, where $\sM(k)$ denotes as before the category of effective Chow motives; to agree with the present convention and the one in \cite{incond}, the first functor is contravariant and the second one is covariant. For $X\in \Sm^\proj(k)$, we write $h_\mot(X)$ for the image in $\Mot(k)$ of $h(X)\in \sM(k)$, in order to avoid confusion.

The fibre functor $\omega$ (resp. $\omega\circ \H$) defines a Tannakian group $\MT_\red$ (resp. $\Gal$): the absolute Mumford-Tate group (resp. the absolute motivic Galois group  \cite[4.6]{incond}); the first one is the largest proreductive quotient of the group $\MT$ considered (for all mixed Hodge structures) in \S \ref{s.hodge}. The functor $\H$ induces a homomorphism of proreductive groups
\[\MT_\red\to \Gal.\]

Let $M\in \Mot(k)$. Write $\Gal(M)$ (resp.  $\MT(M)$) for the image of $\Gal$ (resp. of $\MT_\red$) in $\GL(H(M))$, where $H=\omega\H$. The above homomorphism induces a corresponding monomorphism of reductive groups (not necessarily connected on the right hand side)
\begin{equation}\label{eq12.2}
\MT(M)\inj \Gal(M).
\end{equation}

\begin{warn} The first functor of \eqref{eq12.1} is monoidal but not symmetric monoidal, since the commutativity constraint has been modified in $\Mot(k)$ \cite[4.3]{incond} but not in $\sM(k)$. We shall only use its structure as a functor, not as a $\otimes$-functor.
\end{warn}

We need a further definition:

\begin{defn}[\protect{\cite[6.1]{incond}}]\label{d11.2} An object $M\in \Mot(k)$ is \emph{of 
 abelian type} if $M$  belongs to the Tannakian subcategory $\Mot_\abb(k)$ generated by motives of abelian varieties and Artin motives. A smooth projective variety $X$ is \emph{of motivated abelian type} if $h_\mot(X)$ is of  abelian type.
\end{defn}

\begin{exs}\label{e12.1} a) Any abelian variety, or of any product of curves, is of motivated abelian type. Similarly for Fermat varieties \cite{kat-shi}.  A more recent example is the Fano surface of lines of a cubic 3-fold \cite{diaz}.  (All this is already true in $\sM(k)$.)\\
b) A K3 surface (hence an Enriques surface) is of motivated abelian type \cite[Th. 7.1]{incond}. So is a smooth cubic hypersurface of dimension $\le 6$ (loc. cit., Th. 7.2), and a surface of general type verifying $p_g=K^2=1$ \cite[Cor. 1.5.2]{sht}.\\
c) Let $f:X\to Y$ be a surjective morphism. If $X$ is of motivated abelian type, so is $Y$ ($h_\mot(Y)$ is a direct summand of $h_\mot(X)$.)\\
d) To be of motivated abelian type is a birationally invariant property for smooth projective varieties of dimension $\le 3$.\\
e) A motive $M\in \Mot(k)$ is in $\Mot_\abb(k)$ if and only if it is a direct summand of $N\otimes  h_\mot(A)(r)$ for some abelian variety $A$, some Artin motive $N$ and some $r\ge 0$. (Hint: if $A$ and $B$ are abelian varieties, $h_\mot(A)$ and $h_\mot(B)$ are direct summands of $h_\mot(A\times B)$; the Artin motive $\L$ is a direct summand of $h(E)$ for any elliptic curve $E$.)
\end{exs}

The main properties of $\Mot_\abb(k)$ that we shall use are the following:

\begin{thm}\label{t12.2} Suppose $k$ algebraically closed.\\
 a) In \eqref{eq12.1}, the restriction of $\H$ to $\Mot_\abb(k)$ is fully faithful. In particular, for $M\in \Mot_\abb(k)$, any direct summand of $\H(M)$ is of the form $\H(N)$, where $N$ is a direct summand of $M$.\\
b)  For $M\in \Mot_\abb(k)$, the homomorphism \eqref{eq12.2} is an isomorphism.
\end{thm}

\begin{proof} By rigidity, a) is a translation of \cite[Th. 0.6.2]{incond} (any Hodge cycle on an abelian variety is motivated), and b) follows from a).
\end{proof}

\subsection{Hodge variations and motivated variations}\label{s12.3} Let $X$ be a smooth connected $k$-curve. It is convenient to use the category $\Mot(X)$ of relative Andr\'e motives introduced by Arapura and Dhillon in \cite{ad} (and denoted by $M_A(X)^{str}$ in loc. cit.): its construction relies on André's deformation theorem for motivated cycles \cite[Th. 0.5]{incond}. A family of (motivated) motives in the sense of \cite[\S\S\ before Th. 5.2]{incond} defines an object of $\Mot(X)$. This category is still $\Q$-linear, Tannakian, semi-simple and provided with a faithful, exact $\otimes$-functor
\[\Mot(X)\by{\H} \sV_\Z(X_\C).\]

Indeed, this is part of \cite[Th. 2.3]{ad}, except for the fact that $\H(M)$ is good and of integral origin for any $M\in \Mot(X)$. The first fact follows from the polarizability of $\H(M)$ (cf. \cite[(1) p. 9]{andrecomp}). For the second fact, write $M$ as a direct summand of $h_\mot(\sX)(n)$, where $f:\sX\to X$ is smooth projective and $n\ge 0$. If $p$ is the corresponding projector, then $\H(M)$ is a direct summand via $\H(p)$ of $\bigoplus_{i\ge 0} R^if_*\Q(n)=(\bigoplus_{i\ge 0} R^if_*\Z(n))\otimes \Q$. If $d$ is a positive integer such that $dp$ has integral coefficients, then $\H(M)=\IM(\H(dp))\otimes \Q$.

The functor $\sM(k)\to \Mot(k)$ of \eqref{eq12.1} similarly generalises to a functor
\begin{equation}\label{eq12.1X}
\sM(X)\to \Mot(X)
\end{equation}
where $\sM(X)$ is as in \S \ref{s8.3}.

Suppose $k=\C$; let $M\in \Mot(X)$ and $V=\H(M)$. With notation as before Corollary \ref{tandre}, we write  $T_x(M),\MT_x(M),\Pi_x(M)$ for $T_x(V),\MT_x(V),\Pi_x(V)$. To this notation, we add $\Gal_x(M):=\Gal(M_x)$ and define

\begin{defn} \label{d12.2}
\begin{align*}
\Exc_\H(M)&=\Exc(\H(M)) \text{ (see Cor. \ref{tandre} c));} \\
\Exc_\mot(M)&=\{x\in X(\C)\mid \Pi_x(M)^0\not\subseteq \Gal_x(M)  \}\\
\Exc^N_\H(M)&=\Exc^N(\H(M)) \text{ (see Prop. \ref{p10.1})}\\
\Exc^N_\mot(M)&=\{x\in X(\C)\mid N^1 H(M_x)\ne 0  \}.
\end{align*}
\end{defn}

By \eqref{eq12.2}, we have
\begin{equation}\label{eq12.3}
\Exc_\mot(M)\subseteq \Exc_\H(M).
\end{equation}

Similarly
\begin{equation}\label{eq12.4}
\Exc^N_\mot(M)\subseteq \Exc^N_\H(M)
\end{equation}
since $N^1H(M)\subseteq N^1_\H H(M)$. Finally, by Proposition \ref{p10.1}:
\begin{equation}\label{eq12.5}
\Exc^N_\H(M)\subseteq \bigcup_S \Exc(S)
\end{equation}
where $S$ runs through the simple summands of $\H(M)$, provided $\H(M)$ is connected (see Definition \ref{d10.1} for connected). 

\begin{defn}\label{d12.1} An object $M\in \Mot(X)$ is \emph{of abelian type} if $M_x$ is of abelian type for every $x\notin \Exc_\mot(M)$.\footnote{I don't know if one $x$ is sufficient.} We write $\Mot_\abb(X)$ for the corresponding full subcategory of $\Mot(X)$.
\end{defn} 

\begin{thm}\label{t12.1} Let $M\in \Mot_\abb(X)$ and $N\in \Mot_\abb(\C)$. Suppose $\H(M)$ absolutely simple and connected.  Then
\[\Exc^N_\mot(M\otimes \pi^*N)\subseteq \bigcup_S \Exc_\mot(M\otimes \pi^*S)\]
where $S$ runs through the simple summands of $N$.
\end{thm}

\begin{proof} Clearly, $\H(M\otimes\pi^*N)$ is connected (it has the same monodromy as $\H(M)$). In view of \eqref{eq12.4}, \eqref{eq12.5} and Proposition \ref{p9.2} a), it therefore suffices to show that
\[\bigcup_S \Exc_\mot(M\otimes \pi^*S)=\bigcup_\Sigma \Exc(\H(M)\otimes \pi^*\Sigma)\]
where $\Sigma$ runs through the simple summands of $\H(N)$. By Theorem \ref{t12.2} a), any $\Sigma$ is of the form $\H(S)$ for some $S$. For such an $S$, let $x\notin \Exc_\mot(M\otimes \pi^*S)$. By  Theorem \ref{t12.2} b), $\MT_x(M\otimes \pi^*S)=\Gal_x(M\otimes \pi^*S)$, hence $x\notin \Exc_\H(M\otimes \pi^*S)$ and we are done by \eqref{eq12.3}.
\end{proof}

\subsection{$l$-adic variations}\label{s12.4} We come back to the case where $k$ is an arbitrary subfield of $\C$. Write $\bar k$ for its algebraic closure in $\C$. For $M\in\Mot(X)$, we define
\begin{align}
\Exc_\mot(M)&=\Exc_\mot(M_\C)\cap X(\bar k)\label{eq12.6}\\
\Exc^N_\mot(M)&=\Exc^N_\mot(M_\C)\cap X(\bar k)\notag
\end{align}
cf. Definition \ref{d12.2}. For $d\ge 1$, we also write 
\[X(\bar k)^{\le d}=\{x\in X(\bar k)\mid [k(x):k]\le d\}\] 
and 
\begin{align}
\Exc_\mot(M)^{\le d}&=\Exc_\mot(M)\cap X(\bar k)^{\le d}\label{eq12.6a}\\ 
\Exc^N_\mot(M)^{\le d}&=\Exc^N_\mot(M)\cap X(\bar k)^{\le d}\notag.
\end{align}

By \cite[Th. 2.3]{ad}, there is another realisation functor
\begin{equation}\label{eq12.9}
R_l^\mot:\Mot(X)\to \tilde S_c(X,\Q_l) 
\end{equation}
where $\tilde S_c(X,\Q_l)$ is as in \S \ref{s9.2}. The tannakian theory of this category is parallel to that of $\sV(X_\C)$ as developed in \S \ref{s.hodge}: the r\^ole of $\Loc(X)$ is played by $S_c(X,\Q_l)$ (see \eqref{eq9.11}); for $\bar x$ a geometric point of $X$, the proalgebraic $\Q$-groups $\Pi_x,T_x$ and $\MT$ are replaced by the pro-algebraic $\Q_l$-groups $\Pi_{\bar x}^l,T_{\bar x}^l$ and $\pi^l$ which are the Zariski envelopes (for $l$-adic representations) respectively of the geometric fundamental group $\pi_1^\et(X_{\bar k},\bar x)$, the absolute fundamental group of $\pi_1^\et(X,\bar x)$ and $G_k:=\Gal(\bar k/k)$, with a similar exact sequence to \eqref{eq1V}.\footnote{We write in this subsection $\pi_1^\et$ for the étale fundamental group, in order to avoid confusion with the topological fundamental group $\pi_1$ of \S \ref{s.hodge}.} If we write similarly $\pi^l(k(x),\bar x)$ for the Zariski envelope of $\pi_1^\et(\Spec k(x),\bar x)$, the homomorphism $x_\sV$ is replaced by a homomorphism ${\bar x}_S:\pi^l(k(x),\bar x)\to T_{\bar x}^l$, where $x$ is the point of $X$ underlying $\bar x$. 

To be more specific, here is an $l$-adic analogue to Corollary \ref{tandre}, where $V\in S_c(X,\Q_l)$:

\begin{prop}\label{landre} a) The collections $(T^l_x(V))_{x\in X}$ and $(\Pi^l_x(V))_{x\in X}$ define local systems on $X_\et$. 
Writing $\pi_x^l(V)$ for the image of ${\bar x}_S$ in $GL(V_x)$,  we have 
\begin{align}
\Pi_x(V)\pi^l_x(V) & \text{ is open in }T_x(V), \text{ with equality if } x\in X(k) \label{eq10.4l}\\
\Pi_x(V)^0\pi^l_x(V)^0&=T_x(V)^0\label{eq10.5l}
\end{align}
b) For $x\in X_{(0)}$, the following are equivalent:
\begin{thlist}
\item $\Pi^l_x(V)^0\subseteq  \pi^l_x(V)$. 
\item $\pi^l_x(V)$ is open in $T^l_x(V)$.
\end{thlist}
\end{prop}

The functor \eqref{eq12.9} defines a homomorphism
\[\pi^l\to \Gal\otimes_\Q \Q_l\]
inducing a monomorphism for $M\in \Mot$ and $x\in X_{(0)}$:
\begin{equation}\label{eq12.2l}
\pi^l_x(M)\inj \Gal_x(M)\otimes_\Q \Q_l
\end{equation}
analogous to \eqref{eq12.2}. Since $\pi_1^\et(X_{\bar k})$ is canonically isomorphic to the profinite completion of $\pi_1(X(\C))$, this yields an inclusion, for $M\in \Mot(X)$:
\begin{equation}\label{eq12.7}
\Exc_\mot(M)\subseteq \Exc_l(R_l(M))
\end{equation}
where, for $V\in S_c(X,\Q_l)$
\begin{align}
\Exc_l(V)&=\{x\in X(\bar k)\mid \Pi^l_x(V)^0\not\subseteq \pi^l_x(V)\}\label{eq12.8}\\
&=\{x\in X(\bar k)\mid \pi^l_x(V)\text{ is not open in } T_x^l(V)\}\notag
\end{align}
cf. Proposition \ref{landre} b). This yields:

\begin{thm} \label{t12.4} Suppose that $k$ is finitely generated over $\Q$.\\
a) \cite[Th. 5.1]{cadoret} Let $M\in \Mot(X)$. Then $\Exc_\mot(M)^{\le d}$ is finite for any $d\ge 1$. \\
b) Let $M\in \Mot_\abb(X)$ and $N\in \Mot(k)$; assume that  $N_{\bar k}\in \Mot_\abb(\bar k)$. Then the set $\Exc^N_\mot(M\otimes \pi^*N)^{\le d}$ is finite for any $d\ge 1$. 
\end{thm}

\begin{proof} By \cite[Th. 5.8]{ct1} and \cite[Th. 1.1]{ct2}, the sets $\Exc_l(R_l(M))^{\le d}$ are finite (see next subsection for more details), hence a) follows from \eqref{eq12.7}. Theorem \ref{t12.1} then implies b).
\end{proof}

\subsection{GLP representations}\label{s12.5} In this subsection, we recall results of Cadoret-Tamagawa which were just used in the proof of Theorem \ref{t12.4}, and prove two lemmas which will be used in the next subsection.

\begin{defn}[\protect{\cite{ct1}}] Let $X$ be a scheme separated, geometrically connected and of finite type over $k$. An $l$-adic representation   $\rho: \pi_1^\et(X)\to  GL_m(\Q_l)$ of the étale fundamental group of $X$ is \emph{geometrically Lie perfect} (GLP) if the Lie algebra of the $l$-adic Lie group $\rho(\pi_1(X_{\bar k}))$ is perfect (i.e., its abelianisation is $0$). 
\end{defn}

With the notation of \eqref{eq12.8}, we have by \cite[Th. 1.1]{ct2}:

\begin{thm}\label{t12.5} The set $\Exc_l(\rho)^{\le d}$ is finite for all $d\ge 1$ if $\rho$ is GLP and $k$ is finitely generated over $\Q$. 
\end{thm}

By \cite[Th. 5.8]{ct1}, the monodromy action of $\pi_1^\et(X)$ on $R^*f_*\Q_l$ is GLP if $\dim X=1$ and $f:Y\to X$ is a smooth proper morphism. This explains the proof of Theorem \ref{t12.4} a) in a more detailed way.

We shall need:

\begin{lemma}[see also \protect{\cite[Lemma 2.4]{cadoret2}}]\label{t12.3} Let 
\[0\to \underline{E}\to V\to \pi^* W\to 0\]
be an extension in $S_c(X,\Q_l)$, where $W\in S_c(k,\Q_l)$. Assume that $\underline{E}$ is GLP and semi-simple, and that  the fixed points of $\underline{E}$ under any open subgroup of the monodromy group are trivial.  Then $V$ is GLP.
\end{lemma}

\begin{proof} Fix $x\in X$. Let $G_{\underline{E}}$ (resp. $G_V$) be the monodromy group of $\underline{E}$ (resp. $V$) at $x$. We have an exact sequence
\[
1\to N\to G_V\by{p} G_{\underline{E}}\to 1 
\]
where $N$ embeds into $\Hom(\pi^*W, \underline{E})$. In particular, $N$ is abelian, $G_{\underline{E}}$ acts on $N$ and the properties of the representation $\underline{E}$ carry over to $N$. Abelianising the short exact sequence of Lie algebras
\[0\to N=\Lie(N)\to \Lie(G_V)\to \Lie(G_{\underline{E}})\to 0 \]
we therefore get another short exact sequence
\[0=N^{\Lie(G_{\underline{E}})}\iso N_{\Lie(G_{\underline{E}})}\to \Lie(G_V)^\abb\to \Lie(G_{\underline{E}})^\abb=0.\]

Hence $\Lie(G_V)^\abb=0$, as desired.
\end{proof}

\begin{lemma}\label{l12.3} Let $\rho: \pi_1^\et(X)\to  GL_m(\Q_l)$ be any $l$-adic representation, and let $K$ be an extension of $k$. Then $\Exc_l(\rho)=\Exc_l(\rho_K)$. Here $\rho_K: \pi_1^\et(X_K)\to  GL_m(\Q_l)$ is the representation deduced from $\rho$ via the projection $\pi_1^\et(X_K)\to \pi_1^\et(X)$.
\end{lemma}

\begin{proof} We may assume $K/k$ finitely generated. Write $K=k(U)$, where $U$ is smooth, and let $\bar K$ be an algebraic closure of $K$. Then a point $x\in X(\bar K)$ spreads to a $k$-morphism $\phi:V\to X$, where $V$ is étale over $U$. If $x\notin X(\bar k)$, $\phi$ is dominant (since $\dim X=1$), hence $\IM(\phi_*:\pi_1^\et(V)\to \pi_1^\et(X))$ is open and $x\notin \Exc_l(\rho)$. 
\end{proof}

\subsection{Finiteness}\label{s12.6} We go back to the situation of Subsection \ref{s12.4}. We assume $S_{\bar k}$ to be of motivated abelian type.

\begin{prop}\label{p11.2} Suppose $k$ is finitely generated over $\Q$. Write $M_\mot(\psi_U)$ for the image of $M(\psi_U) \in \sM(U)$ in $\Mot(U)$ by the functor of \eqref{eq12.1X} (see \eqref{eq11.1} for $M(\psi_U)$). Then the set $\Exc^N_\mot(M_\mot(\psi_U))^{\le d}$ \eqref{eq12.6a}
is finite for any $d\ge 1$.
\end{prop}

\begin{proof} This is the special case $M=h^1_\mot(\psi_U)$, $N=t^2_\mot(S\times U/U)$ of Theorem \ref{t12.4} b).
\end{proof}

\begin{prop}\label{p12.2} a) The local system $R_l(\widehat{\psi_U})\in \tilde S_c(U,\Q_l)$ is GLP.\\
b) For any $d\ge 1$, the set
\[\{u\in U(\bar k)\mid [k(u):k]\le d\text{ and  }R_l(\widehat{\psi_u})=0\} \]
is finite.
\end{prop}

\begin{proof} The proof of a) is analogous to that of Corollary \ref{c12.1}; we apply  Lemma \ref{t12.3} to the extension $\sE(\psi)$ defined by $R_l(\widehat{\psi_U})$: 
\[0\to H^1_l(\psi_U)\to \sE(\psi)\to \pi^*H^2_{l,\tr}(S)\to 0.\]

By Proposition \ref{h10.1} a), the left term is isomorphic to the local system $\underline{E}$ of vanishing cycles, and the hypothesis of Lemma \ref{t12.3} is verified by part  b) of this proposition. (Actually, $\underline{E}$ remains irreducible when restricted to any open subgroup of the monodromy.) Therefore $\sE(\psi)$ is GLP. b) now follows from a), Theorem \ref{t12.5} and Theorem  \ref{c10.1}.
\end{proof}

\begin{thm}\label{t11.3} Suppose $k$ finitely generated over $\Q$. Then, for any $d\ge 1$, the set
\begin{multline*}
\Exc(\psi)=\{u\in U(\bar k)\mid [k(u):k]\le d\text{ and  }\widehat{\psi_u}=0 \text{ in }\\
 \sM_\alg(k(u))(h_1(\psi_u),t_2(S_{k(u)})\} 
\end{multline*}
is finite. So is the corresponding set for the condition
\[(\psi_u)_\#=0 \text{ in } \Griff(C_u\times_k S)\otimes \Q.\]
Let $K$ be a field containing $\bar k$ (e.g., $K=\C$). With notation as in Theorem \ref{t12.6}, the set $\Exc^\sM(\psi_K)$ is contained in $U(\bar k)$; in particular, it is countable.
\end{thm}

\begin{proof} This follows from Propositions \ref{p11.2}, \ref{p12.2} and  \ref{p9.4} and (for the last statement) from Lemma \ref{l12.3}.
\end{proof}

\subsection{Example: K3 surfaces}

\begin{cor}\label{c12.2} In Theorem \ref{t11.3}, assume that $S$ is a K3 surface and let $T_u=C_u\times S$. Then, for $u\notin \Exc(\psi)$, Murre's algebraic representative $\ab^2(T_u)$ is isogenous to $J(C_u)\otimes \NS_S$.
\end{cor}

\begin{proof} Since $h_1(S)=h_3(S)=0$, we have $h_1(C_u)=h_1(\psi_u)$ and
\[H^3_l(T_u)=H^1_l(\psi_u)\otimes H^2_l(S)\]
hence
\[N^1H^3_l(T_u)=H^1_l(\psi_u)\otimes \NS_S=H^1_l(C_u)\otimes \NS_S\]
and the claim. 
\end{proof}

\newpage

\appendix

\section{A direct proof of Theorem \ref{T1}}\label{A}

Let $X\in \Sm^\proj(k)$. Recall from \cite[Def. 1.5]{ram} the Albanese scheme $\sA_X$ of $X$: it is a $k$-group scheme locally of finite type, sitting in an extension
\begin{equation}\label{eqA.1}
0\to \Alb_X\to \sA_X\to \Z\pi_0(X)\to 0
\end{equation}
where $\Alb_X$ is the Albanese variety of $X$ and $\pi_0(X)$ is the scheme of constants of $X$. In \cite[(8.1)]{birat-pure}, we aggregated the degree map and the Albanese map from the Chow group of $0$-cycles of $X$ into a single homomorphism
\begin{equation}\label{eqA.3}
\begin{CD}
CH_0(X)@>a_X^k>> \sA_X(k).
\end{CD}
\end{equation}

We write
\[T(X)=\Ker a_X^k\]
(the Albanese kernel).

Let $Y$ be another smooth projective $k$-variety. As in  \cite[(8.1.3)]{birat-pure} the homomorphism \eqref{eqA.3} and the universal property of $\sA_X$ yields a map
\begin{equation}\label{eq5}
CH_0(X_{k(Y)})\by{a_{Y,X}} \sA_X(k(Y))= \Hom(\sA_Y,\sA_X).
\end{equation}

By \cite[Prop. 8.2.1]{birat-pure}, \eqref{eq5} makes $X\mapsto \sA_X$ a functor from the category $\Chow^\o$ of birational Chow motives (with integral coefficients) to the category $\AbS$ of abelian schemes; a fortiori it is a functor on $\Chow^\eff$. In other words, the composite map
\begin{equation}\label{eq5a}
CH^{\dim X}(Y\times X) \to CH_0(X_{k(Y)})\by{a_{Y,X}} \Hom(\sA_Y,\sA_X)= \sA_X(Y)
\end{equation}
\cite[(8.1.2)]{birat-pure} is compatible with the composition of correspondences. 

The morphism \eqref{eqA.3} is almost split. Namely, in \cite[Prop. 8.2.1 and Th. 8.2.4]{birat-pure} we prove that the functor $a_\Q:\Chow^\o(k,\Q)\to \AbS(k,\Q)$ is full, essentially surjective and has a fully faithful right adjoint $\rho$, whose essential image is the thick subcategory $\Chow^\o_{\le 1}(k,\Q)$ of \break $\Chow^\o(k,\Q)$ generated by birational motives of curves. This yields:

\begin{prop}\label{pA.1} Let $X\in \Sm^\proj(k)$. Then there exist an integer $n>0$ and, for any connected $Y\in \Sm^\proj(k)$, a homomorphism $\sigma_Y:\sA_X(k(Y))\to CH_0(X_{k(Y)})$ such that
\begin{thlist}
\item $\sigma_Y$ is natural in $Y$ for the action of correspondences (in $\Chow^\o$).
\item $a_X^{k(Y)}\sigma_Y = n$ for any $Y$.
\end{thlist}
\end{prop}

\begin{proof} Let $M=\rho(\sA_X)$. The unit map of the adjunction yields a morphism $\pi:h^\o(X)\to M$,  where $h^\o(X)\in \Chow^\o(k,\Q)$ is the birational motive of $X$, which induces an isomorphism $a_\Q(\pi):\sA_X\iso \sA_M$ in $\AbS(k,\Q)$. By the fullness of $a_\Q$, there exists a morphism $\sigma_0:M\to h^0(X)$ such that $a_\Q(\sigma_0)$ is the inverse of $a_\Q(\pi)$. Write $M$ as a direct summand of the birational motive of a (not necessarily connected) curve $C$. Then $\sigma_0$ is induced by an algebraic correspondence from $C$ to $X$ with rational coefficients. There is an integer $n_1>0$ such that $\sigma_1=n_1\sigma_0$ has integral coefficients. Since $CH_0(C_K)\to \sA_C(K)$ is injective with cokernel killed by some universal integer $n_2$ (see \cite[Lemma 8.2.3]{birat-pure}), $\sigma=n_2\sigma_1$ defines the desired system $(\sigma_Y)$.
\end{proof}

\begin{cor} a) There is an element $\xi\in T(X_{k(X)})$ with the following property: for any $Y\in \Sm^\proj$ and any $y\in T(X_{k(Y)})$, there is a morphism $y^*:T(X_{k(X)})\to T(X_{k(Y)})$ such that $y^*\xi = ny$. Here $n$ is as in Proposition \ref{pA.1}.\\
b) Suppose that $T(X_\Omega)=0$, where $\Omega$ is a universal domain. Then there is an integer $m>0$ such that $mT(X_{k(Y)})=0$ for any $Y\in \Sm^\proj(k)$.
\end{cor}

\begin{proof} Viewing $y$ as an element of $CH_0(X_{k(Y)})=\Chow^\o(h^\o(Y),h^\o(X))$, it defines by pull-back a morphism $y^*:CH_0(X_{k(X)})\to CH_0(X_{k(Y)})$, which sends $T(X_{k(X)})$ into $T(X_{k(Y)})$ by the naturality of $a_X$. Let $\eta\in CH_0(X_{k(X)})$ be the class of the generic point: note that $y^*\eta=y$ (view $\eta$ as the identity endomorphism of $h^\o(X)$). Define $\xi = n\eta- \sigma_Xa_X^{k(X)} (\eta)$ where $\sigma_X$ is as in Proposition \ref{pA.1}. Then $\xi\in T(X_{k(X)})$ and
\[y^*\xi = ny - \sigma_Y a_X^{k(Y)} y=ny\]
by the naturality of $a_X$ and $\sigma$.

b) The hypothesis implies that $\xi$ is torsion, hence the conclusion follows from a).
\end{proof}

\begin{rk} As a converse to b), it follows from a) and Ro\v\i tman's theorem that $T(X_\Omega)=0$ if $\xi$ is torsion. Unfortunately, $\xi$ is not canonically defined (it depends on the choice of the quasi-section $\sigma$ of Proposition \ref{pA.1}). It would be interesting to compute the exact exponent of $T(X_{k(Y)})$ when $X$ is a surface  with $q>0$ verifying Bloch's conjecture, as an extension of \cite{tbm}.
\end{rk}

Composing \eqref{eq5} with restriction to the connected components, we get a map
\begin{equation}\label{eqA.2}
CH^{\dim X}(Y\times X) \to \Hom(\Alb_Y,\Alb_X)
\end{equation}
which is induced by the action of correspondences, hence factors through numerical equivalence. We have:

\begin{lemma}\label{lA.2} a) (cf. \protect{\cite[(8.1.5)]{birat-pure}}) The restriction map $\Hom(\sA_Y,\sA_X)\allowbreak \to \Hom(\Alb_Y,\Alb_X)$ sits in an exact sequence
\[0\to \Hom(\sA_{\pi_0(Y)},\sA_X)\to \Hom(\sA_Y,\sA_X)\to \Hom(\Alb_Y,\Alb_X).\]
b) We have a complex
\[0\to \frac{T(X_{k(Y)})}{T(X_{\pi_0(Y)})}\by{b} \frac{CH_0(X_{k(Y)})}{CH_0(X_{\pi_0(Y)})}\by{a} \Hom(\Alb_Y,\Alb_X)\to 0\]
which is split exact up to the integer $n$ of Proposition \ref{pA.1}.
\end{lemma}

The statement of b) means that there are maps $r,s$ such that $r b=n$ and $as = n$.

\begin{proof} For a), the exact sequence  yields an exact sequence
\[0\to \Hom(\sA_{\pi_0(X)},\sA_Y)\to \Hom(\sA_X,\sA_Y)\to \Hom(\Alb_X,\sA_Y) \]
and the latter group is isomorphic to $\Hom(\Alb_X,\Alb_Y)$. Then b) follows by a diagram chase.
\end{proof}

Assume now that $Y$ is a curve $C$, that for simplicity we further suppose geometrically connected.  The localisation exact sequences for the inclusions $U\times X\inj C\times X$, where $U$ runs through the nonempty open subsets of $C$, yield in the limit an exact sequence
\[\bigoplus_{c\in C_{(0)}} CH^{d-1}(X_{k(c)})\to CH^d(C\times X)\to CH^d(X_{k(C)})\to 0\]
(where $d=\dim X$), which itself induces an exact sequence
\[\bigoplus_{[E:k]<\infty} CH^1(C_E)\otimes CH^{d-1}(X_E) \to CH^d(C\times X)\to CH^d(X_{k(C)})\to 0\]
where the first map is given by external products and transfers. Together with Lemma \ref{lA.2} b), this yields:

\begin{prop}\label{pA.2} Suppose that $C$ is a geometrically connected curve, and let $d=\dim X$. Then we have a complex
\begin{multline*}
0\to \frac{T(X_{k(C)})}{T(X)}\to \frac{CH^d(C\times X)}{\displaystyle CH^d(X)+\sum_{[E:k]<\infty} \Tr_{E/k}\left (CH^1(C_E)\cdot CH^{d-1}(X_E)\right)}\\
\by{a'} \Hom(\Alb_C,\Alb_X)\to 0
\end{multline*}
which is split exact up to the integer $n$ of Proposition \ref{pA.1}.\qed
\end{prop}

The proof of the next lemma is routine and omitted.

\begin{lemma}\label{lA.1} Write $\pi:X\times C\to X$ for the projection. Assume that $C$ has a rational point $c$, and write $i:X\to C\times X$ for the corresponding inclusion. Consider the maps
\begin{gather*}
\alpha:CH^d(X)\oplus CH^1(C)\otimes CH^{d-1}(X)\by{(\pi^*,\cup)} CH^d(C\times X)\\
 \beta:CH^d(C\times X)\by{\begin{pmatrix}i^* \\ \pi_* \end{pmatrix}}CH^d(X)\oplus  CH^{d-1}(X).
\end{gather*}
Then 
\[\beta\circ \alpha = \diag(1_{CH^d(X)}, \deg\otimes 1_{CH^{d-1}(X)}).\]
The same holds a fortiori for coarser adequate equivalence relations.\qed
\end{lemma}

Since \eqref{eqA.2} factors through numerical equivalence, so does $a'$ in Proposition \ref{pA.2}. We have:

\begin{prop}\label{pA.3} The map
\[\frac{A^d_\num(C\times X)}{\displaystyle A^d_\num(X)+\sum_{[E:k]<\infty} \Tr_{E/k}\left (A^1_\num(C_E)\cdot A^{d-1}_\num(X_E)\right)}
\by{\bar a} \Hom(\Alb_C,\Alb_X)\]
induced by the map $a'$  of Proposition \ref{pA.2} has finite kernel and cokernel; its cokernel is killed by the integer $n$ of Proposition \ref{pA.1} and its kernel is killed by
\[I(C)= \gcd\{\deg(c)\mid c\in C_{(0)}\}\]
(the \emph{index} of $C$).
\end{prop}

\begin{proof} We first prove the analogous statement up to isogeny when replacing numerical equivalence with homological equivalence relative to some classical Weil cohomology, and assuming that $C$ has a rational point $c$. Then $A^1_\hom(C_E)=\Z [c]$ for any $E$, so that
\begin{multline*}
\frac{A^d_\hom(C\times X)}{\displaystyle A^d_\hom(X)+\sum_{[E:k]<\infty} \Tr_{E/k}\left (A^1_\hom(C_E)\cdot A^{d-1}_\hom(X_E)\right)} 
\\= \frac{A^d_\hom(C\times X)}{\displaystyle A^d_\hom(X)+[c]\cdot A^{d-1}_\hom(X)}.
\end{multline*}

Consider the commutative diagram
\[\begin{CD}
A^d_\hom(X)\oplus A^1_\hom(C)\otimes A^{d-1}_\hom(X)@>\alpha>> A^d_\hom(C\times X) @>a'>> \Hom(\Alb_C,\Alb_X)\\
@VVV @VVV @VVV\\
H^{2d}(X)\oplus H^2(C)\otimes H^{2d-2}(X)@>\alpha_H>> H^{2d}(C\times X)@>a'_H>> \Hom(H^1(C),H^{2d-1}(X))
\end{CD}\]
where $\alpha$ is as in Lemma \ref{lA.1}, the vertical maps are cycle class maps and $\alpha_H,a'_H$ are constructed like $\alpha$ and $a'$. The lower sequence is short exact, by the K\"unneth formula and Poincaré duality. Recall that $n\Coker(a')=0$ and that $\alpha$ has the retraction $\beta$. To obtain the exactness of the top sequence up to isogeny, it suffices to show that $\Ker a'\cap \Ker \beta$ is torsion. But so is $\Ker \bar a_H\cap \Ker \beta_H$, where $\beta_H$ is  defined analogously to $\beta$, thus the conclusion follows from the injectivity of the cycle class map.

We just proved that the map
\[\frac{A^d_\hom(C\times X)}{\displaystyle A^d_\hom(X)+[c]\cdot A^{d-1}_\hom(X)}\by{\bar a}  \Hom(\Alb_C,\Alb_X)\]
has torsion kernel. A fortiori, so does the corresponding map for numerical equivalence. But $A^d_\num(C\times X)$ is torsion-free, hence its quotient $\frac{A^d_\num(C\times X)}{A^d_\num(X)+[c]\cdot A^{d-1}_\num(X)}$ by a direct summand is torsion-free too, and the map $\bar a$ of Proposition \ref{pA.3} is injective in this case. The general case now follows by a transfer argument.
\end{proof}

\begin{rk}
As a byproduct of this proof, we see that the left hand side of the map in Proposition \ref{pA.3} coincides up to torsion for homological and numerical equivalence.
\end{rk}

Let $CH^*_\num:=\Ker(CH^*\to A^*_\num)$. Putting Propositions \ref{pA.2} and \ref{pA.3} together, we get:

\begin{cor}\label{cA.1} If $C$ is a smooth connected curve having a $0$-cycle of degree $1$, the map of Proposition \ref{pA.2} induces a homomorphism
\[\frac{T(X_{k(C)})}{T(X)}\to \frac{CH^d_\num(C\times X)}{\displaystyle CH^d_\num(X)+\sum_{[E:k]<\infty} \Tr_{E/k}\left (\Pic^0(C_E)\cdot CH^{d-1}(X_E)\oplus \Pic(C_E)\cdot CH^{d-1}_\num(X_E)\right)}\]
with kernel and cokernel killed by $n$.\qed
\end{cor}

Note that all groups in the denominator of the right hand side, except perhaps $\Pic(C_E)\cdot CH^{d-1}_\num(X_E)$, are made of cycles algebraically equivalent to $0$; this is also true for the latter when $d=2$. Hence

\begin{cor}\label{cA.2} For $d=2$, the map of Corollary \ref{cA.1} induces a homomorphism with cokernel killed by $n$
\[\frac{T(X_{k(C)})}{T(X)}\to \Griff(C\times X). \quad \quad\text{\qed} \]
\end{cor}

In case $C$ does not have a $0$-cycle of degree $1$, one may have to multiply the map of Proposition  \ref{pA.2} by $I(C)$ to obtain a map like the one in Corollary \ref{cA.2}.

\begin{cor}\label{cA.3} Let $d=2$, and let $\Omega/k$ be a universal domain. If $T(X_\Omega)=0$, then $nI(C)\Griff(C\times X)=0$ for any curve $C$, where $n$ is as in Proposition \ref{pA.1}.\qed
\end{cor}

\begin{proof}[Proof of Corollary \ref{C2}] Let $T=C\times S$: by Soulé \cite[Th. 4.1]{soule} or \cite[Th. 7.5.1]{kcag}, $CH^2(T)\otimes \Q\iso A^2_\num(T)$; by \cite[Cor. 7.5.3]{kcag}, $CH^2(T)$ is finitely generated. It follows that $CH^2_\num(T)$ is finite and that $T(S_{k(C)})$ is finitely generated as a subquotient of $CH^2(T)$. Corollary \ref{cA.1} then implies that $T(S_{k(C)})/T(S)$ is finite. But $T(S)$ is finite, for example by Kato-Saito \cite{kasa}. This concludes the proof.
\end{proof}

\section{The Chow-Lefschetz isomorphism for $C\times S$}

\begin{thm}\label{tB.1} Let $C,S$ be a curve and a surface over $k$, that we both assume geometrically connected, and let $T=C\times S$. Let $L'\in \Pic(C)$, $L''\in \Pic(S)$ be classes of smooth hyperplane sections, and $L =L'\otimes 1+1\otimes L''\in \Pic(T)$ be the class of the hyperplane section attached to the corresponding Segre embedding. Then the morphism
\[h_4(T)\by{\cdot L} h_2(T)(1)\]
is an isomorphism.
\end{thm}

\begin{proof} Write $h_2(S)=h'_2(S)\oplus p_2(S)$, where $p_2(S)=\Ker(h_2(S)\by{L''} h_0(S)(1) = \Coker(h_4(S)(-1)\by{L'} h_2(S)$ (primitive part). Then we get decompositions of $h_4(T)$ and $h_2(T)$ into $4$ summands:
\begin{align*}
h_4(T)&=h_2(C)\otimes p_2(S) \oplus h_2(C)\otimes h'_2(S) \oplus h_1(C)\otimes h_3(S) \oplus h_0(C)\otimes h_4(S)\\
h_2(T)&=h_2(C)\otimes h_0(S)\oplus h_1(C)\otimes h_1(S)\oplus h_0(C)\otimes p_2(S)\oplus h_0(C)\otimes h'_2(S).
\end{align*}

The matrices of $L'\otimes 1$ and $1\otimes L''$ with respect to these decompositions are respectively
\[
\begin{pmatrix}
0&0&0&0\\
0&0&0&0\\
i'&0&0&0\\
0&0&0&0
\end{pmatrix}
\qquad
\begin{pmatrix}
0&i''_1&0&0\\
0&0&0&0\\
0&0&i''_2&0\\
0&0&0&i''_3
\end{pmatrix}
\]
where $i':h_2(C)\otimes p_2(S)\to h_0(C)(1)\otimes p_2(S)$, $i''_1:h_2(C)\otimes h'_2(S)\to h_2(C)\otimes h_0(S)(1)$, $i''_2:h_1(C)\otimes h_3(S)\to h_1(C)\otimes h_1(S)(1)$ and $i''_3:h_0(C)\otimes h_4(S)\to h_0(C)\otimes h'_2(S)(1)$ are the isomorphisms induced by $L'$ and $L''$. The sum of these two matrices is clearly invertible.
\end{proof}

\section{Normality criteria, and monodromy of variations of mixed Hodge structures}\label{yves}

\hfill by Yves Andr\'e\footnote{the author of this appendix thanks B. Kahn for the invitation to write it and for his remarks and suggestions of improvement of the exposition.}
\bigskip

 The aim of this appendix is to review and unravel some exactness criteria for a short sequence of linear algebraic groups, and to apply these criteria to the monodromy of variations of mixed Hodge structures.
 
 \subsection{Normality criteria}
 
 Let 
 \begin{equation}\label{e1A} 1\to H\stackrel{q}{\to} G \stackrel{p}{\to} Q \to 1
 \end{equation} 
 be a sequence of morphisms of linear algebraic groups (= geometrically reduced, {i.e. smooth}, affine group schemes of finite type) over a field $F$ of any characteristic. 
 
\subsubsection{} In many concrete situations, one wishes to establish the exactness of \eqref{e1A} by tannakian duality, \ie using properties of the categories of representations. 
 There are well-known necessary and sufficient conditions for exactness on the left and on the right due to Saavedra (see \eg \cite[2.21]{DM}), namely:
 
\smallskip ${\bf P}1.\,$  $p$ is faithfully flat if and only $p^\ast({\rm{Rep}}_F\, Q)$ is a tannakian subcategory of ${\rm{Rep}}_F\, G$ (equivalently: is a full subcategory closed under taking quotients (or subobjects)), 

\smallskip ${\bf P}2.\,$  $q$ is a closed embedding if and only if any object of ${\rm{Rep}}_F\, H$ is a subquotient of an object of $q^\ast({\rm{Rep}}_F\, G)$.
 
 \medskip
Henceforth, we concentrate on exactness in the middle, assuming that {\it $q$ is a closed embedding} and {\it $p$ is faithfully flat}.  
 
\smallskip For any (rational, finite-dimensional) representation $V\in {\rm{Rep}}_F\, G$, we denote by $V^H$ the $F$-subspace of $H$-invariants. For any character $\chi: H\to \G_{m\,F}$, we denote by $V^\chi$ the $F$-subspace of $V$ where $H$ acts through $\chi$. We say that a character $\chi$ occurs in a representation of $G$ is there exists $V\in {\rm{Rep}}_F\, G$ such that $V^\chi\neq 0$.

  \smallskip  An obvious exactness criterion is 
 
 \begin{lemma}\label{l1A} \eqref{e1A} is exact if and only if the following conditions hold:
 
 ${\bf H}0$. In every $G$-representation $V$, $V^H = V^{\ker p}$,

  ${\bf H}1$. $H$ is a normal subgroup of $G$. 
 \end{lemma}

The ``only if" part is trivial. For the ``if" part, let $\bar H$ denote the smallest normal closed subgroup of $G$ containing $H$; then ${\bf H}0$ implies $\bar H=\Ker p$ (consider a faithful representation of $G/\bar H$), and ${\bf H}1$ is equivalent to $H=\bar H$. 

Condition ${\bf H}0$ alone is {\it not enough} to imply exactness\footnote{this is a common source of errors.}, as the example $Q= 1, \, G=GL_2, \,H= \{$upper triangular matrices$\}$ shows.
 
\smallskip The lemma turns the exactness problem into expressing ${\bf H}1$ in represen\-tation-theoretic terms. 
 We shall discuss two such normality criteria.

 \subsubsection{} The first criterion appears in \cite[Lemma 1]{andrecomp}:
 
 \begin{prop}\label{p1A}  $H$ is normal in $G$ if the following condition holds: 
 
 ${\bf H}2.$ For every character $\chi$ of $H$ occurring in a representation $V$ of $G$, $V^\chi$ is stable under $G$. 
 
 \smallskip (The converse is true if $G$ is connected). 
 \end{prop} 

 This is a consequence of Chevalley's theorem \cite[4.19]{mil}: $H$ is the stabilizer of a line $\ell$ in some representation $V$ of $G$. This line defines a character $\chi$. By assumption, $V^\chi$ is $G$-stable.  Let $K$ be the kernel of $G\to GL(\End_F V^\chi)$. Then $H\subset K$. Conversely, $K$ commutes with any $F$-endomorphism of $V^\chi$, hence acts by scalars on $V^\chi$, hence stabilizes $\ell$. Therefore $H=K$. 
 
 (For the converse, the idea is that $G\times H\stackrel{(g,h)\mapsto ghg^{-1}}{\to} H \stackrel{\chi}{\to} \G_m$ defines a morphism from $G$ to the \'etale-locally constant group scheme $\Hom(H,\G_m)$ defined by a $\Z$-module of finite type; such a morphism factors through a finite quotient, hence is trivial if $G$ is connected.)

\subsubsection{}  The second criterion appears in \cite[Th. A1]{EHS}\footnote{precisely, the special case of \loccit for a sequence \eqref{e1A} in which $Q=G/\bar H$.}: 
  \begin{prop}\label{p2A} $H$ is normal in $G$ if and only if the following conditions hold:
  
   ${\bf H}3.$ For every $V\in  {\rm{Rep}}_F\, G$, $V^H$ is $G$-stable.
   
    ${\bf H}4.$ $H$ is an {\emph{observable}} subgroup of $G$, \ie every representation of $H$ is contained in a representation of $G$. 
 \end{prop} 
 
One aim of this section is to highlight the relevance of the notion and of the theory of observability in the context of normality criteria\footnote{actually, neither this term nor the theory of observability is mentioned in \cite{EHS} (where the fact that a representation of a normal subgroup $H\triangleleft G$ is contained in a representation of $G$ is proved directly) - nor in any of the quoted references below about normality criteria.}.  
 Let us mention several useful characterizations of observability (due to several authors, \cf \cite{Gros} for a general reference where $F$ is algebraically closed, and \cite[Th. 9]{TB} for the reduction to this case):
 
   \begin{prop}\label{p3A} The following are equivalent:
   \begin{enumerate}
   \item $H$ is observable in $G$,
    \item every character of $H$ occurs in a representation of $G$,
    \item $H$ is the isotropy group of some element in a representation of $G$,
     \item $G/H$ is quasi-affine.
     \end{enumerate}
     \end{prop}
     
     \begin{exs}\label{c1A} $(1)$ {\it A normal subgroup is observable} (this follows from the characterization $(4)$, \cf \cite[5.21]{mil}). In particular, $\bar H$ is observable in $G$, and $H$ is observable in $\bar H$ is and only if it is observable in $G$.        
     
\smallskip\item $(2)$ {\it If every character of $H$ occurring in a representation of $G$ is of finite order, then $H$ is observable in $G$} ($H$ is the stabilizer of a line $\ell$ in a $G$-representation $V$; this defines a character, which is assumed to be of finite order $n$; then $H$ is the stabilizer of $\ell^{\otimes n}$ in $V^{\otimes n}$ (using the fact that the Segre diagonal map $\P(V)\to \P(V)^n \to \P(V^{\otimes n})$ is an embedding, so that $\ell^{\otimes n}$ determines $\ell$), hence is the isotropy subgroup of any generator of $\ell^{\otimes n}$, and one concludes by $(3)$). 
     \end{exs}
 
  Let us come back to the proof Proposition \ref{p2A}. The ``only if" part follows from Example \ref{c1A} $(1)$. Let us turn to the ``if" part.  Since $\bar H$ is observable (\ref{c1A} (2)), for any $W,W'\in {\rm{Rep}}_F\, \bar H${, there is an injective morphism $W'\inj V'$ and dually a surjective morphism $V\surj W$ in ${\rm{Rep}}_F\, \bar H$, with $V,V'\in {\rm{Rep}}_F\,  G$. For any  $H$-equivariant morphism $f: W\to W'$,} the {corresponding} composition $g:V\to V'$ is a morphism in ${\rm{Rep}}_F\, \bar H$ if and only if so is $f$ - and this holds true by ${\bf H}3$ applied to $g\in \Hom(V, V')^H$. This shows that ${\rm{Rep}}_F\, \bar H\to {\rm{Rep}}_F\, H$ is full. 
 
Let now $ W$ be a quotient in ${\rm{Rep}}_F\,  H$ of an object $V\in {\rm{Rep}}_F\, \bar H$.  By ${\bf H}4$, $W$ is contained in some $V' \in {\rm{Rep}}_F\,  \bar H$. Thus $W$ is the image of $V\to V'$, which is a morphism in ${\rm{Rep}}_F\, \bar H$ by fullness. By {\bf P}1,  $H\inj \bar H$ is faithfully flat, so that $H=\bar H$.
 
\subsubsection{} Combining \ref{p2A} and \ref{c1A} $(2)$, one obtains the following new normality criterion:

 \begin{prop}\label{p4A}  $H$ is normal in $G$ if both $\,{\bf H}3$ and the following condition hold:
 
 ${\bf H}5.$ every character of $H$ occurring in a representation of $G$ is of finite order.
  \end{prop} 

 \begin{rk}\label{r0A} There are other normality criteria in the literature, for instance \cite[\S 4]{dS} (which is compared to \cite[Th. A1]{EHS} in \cite{LP}). See also \cite[\S 2]{Ber} and \cite{Mi} in the context of non-neutral tannakian categories.
 \end{rk}
 
  \begin{rk}\label{r1A} Propositions \ref{p1A}, \ref{p2A}, \ref{p4A} extend to arbitrary reduced affine group schemes over $F$, since for any object $V\in {\rm{Rep}}_F\, G$, the action of $G$ factors through a finite-dimensional quotient.
   \end{rk}

 \medskip
 \subsection{Application to variations of mixed Hodge structures}

\subsubsection{} Let $S$ be a smooth connected complex algebraic variety, with a base point $s$. We consider {\it good graded-polarizable variations $V$ of mixed $\Q$-Hodge structures over $S^{an}$, whose underlying local system is defined over $\Z$} (\ie there is a lattice in $V_s$ which is stable under $\pi_1(S,s)$), \cf \cite{bez}. They form a {\it tannakian category $\sV_S$}, with fiber functor $\omega_s: \sV_S\to Vec_\Q$ given by the fiber at $s$. We denote by $T_{S,s}$ the corresponding tannakian group. 

The forgetful $\otimes$-functor $ \sV_S \to \Loc_S$ given by the underlying local system of $\Q$-vector spaces gives rise to a morphism of tannakian groups $\Pi_{S,s} \to T_{S,s}$, and we denote by $M_{S,s}$ (for ``monodromy") its image in $T_{S,s}$ (which is also the Zariski closure of the image of $\pi_1(S,s)$ in $T_{S,s}$). 

On the other hand, {\it constant variations} (those for which the local system is constant) form a tannakian subcategory of $\sV_S$, which can be identified via $\omega_s$ with the category ${\bf{MHS}}$ of graded-polarizable 
mixed $\Q$-Hodge structures. We denote by $\MT$ (for ``Mumford-Tate") its tannakian group, and by $p: T_{S,s}\to \MT$ the natural morphism. By ${\bf P}1$, $p$ is faithfully flat.

\subsubsection{} The assumption that local systems are defined over $\Z$ has the following useful consequence:

\begin{lemma}\label{l2A} Characters $\chi$ of $M_{S,s}$ occurring in a representation $V$ of $T_{S,s}$ are of finite order.
 More precisely, $\chi$ is either trivial or quadratic.
\end{lemma} 

Indeed, $V$ has a lattice stable under $\pi_1(S,s)$, so that $V^\chi$ inherits a $\Z$-structure on which $\pi_1(S,s)$ acts through $\chi$. 
Therefore $\chi$ is either trivial or quadratic since $\Z^\times= \{\pm 1\}$. \qed

\subsubsection{} ``Good" (a.k.a. ``admissible") is a condition on the behaviour at infinity of the variation\footnote{This condition is automatically satisfied in the case of pure Hodge structures, and is known to hold for variations of geometric origin. For a simple 
counterexample, inspired by \cite{sz}, when this condition is not fulfilled, see \cite[p.12]{andrecomp}. For a exhaustive discussion, \cf \cite{bez}.},
 which ensures that the so-called ``theorem of the fixed part" holds:

\begin{thm}\label{t1A}\cite{sz} $i)$ For any $V\in {\rm{Rep}}_\Q\,  T_{S,s}$ (corresponding to a variation ${\bf V}\in \sV_S$), the fixed part $W := V^{M_{S,s}}$ is stable under $T_{S,s}$ (\ie is the fiber at $s$ of a subvariation $\bf W$ of $\bf V$).

\noindent $ii)$ Moreover, it is a representation of $\MT$ (\ie $\bf W$ is a constant variation).
\end{thm}

\subsubsection{} The main result of this appendix is the following (which is a variation on \cite[Th. 1]{andrecomp}\footnote{\cite[Th. 1]{andrecomp} deals with the Mumford-Tate group a {\it general} fiber ${\bf{V}}_s$ instead of the tannakian group of the variation ${\bf V}$, and with the connected monodromy group. Actually, it can be shown that the Mumford-Tate group of ${\bf{V}}_s$  coincides with the image of $T_{S,s}^0$ (using \cite[Lemma 4]{andrecomp}, a variation on a result of Deligne in the mixed case), and that \cite[Th. 1]{andrecomp} can be interpreted as the normality of $M_{S,s}^0$ in $T_{S,s}^0$.}
 and is mentioned in \cite[\S 1.4]{A2}):  

\begin{thm}\label{t2A}  The sequence of affine group schemes over $\Q$
\begin{equation}\label{e2A}
1\to M_{S,s}\stackrel{q}{\to} T_{S,s} \stackrel{p}{\to} \MT \to 1
 \end{equation} is split exact.
\end{thm}

\begin{proof} Part $ii)$ of Theorem \ref{t1A} shows that Condition ${\bf H}0$ of Lemma \ref{l1A} holds. To check ${\bf H}1$, we use Proposition \ref{p4A} (and Remark \ref{r1A}): part $i)$ of Theorem \ref{t1A} shows that ${\bf H}3$ holds, and Lemma \ref{l2A} shows that ${\bf H}5$ holds. We conclude that \eqref{e2A} is exact. 
 The $\otimes$-functor ${\bf V}\to {\bf V}_s: \sV_S\to \MHS$ provides a canonical splitting of \eqref{e2A}. 
\end{proof}

\begin{rk}\label{r2A} Instead of Proposition \ref{p4A}, one can alternatively use Proposition \ref{p1A} (which avoids the theory of observability) at the cost of a more delicate argument.
Let $\chi$ be a character occurring in a representation $V$ of $T_{S,s}$ (corresponding to an object ${\bf V}\in \sV_S$). If $\chi$ is trivial,  $V^\chi$ is stable under $T_{S,s}$ by Theorem \ref{t1A}. Otherwise, by Lemma \ref{l2A}, there is a double \'etale covering $S'\to S$ which kills $\chi$. Let $s'$ be one of the two points of $S'$ above $s$, and let us look at the pull-back ${\bf V}_{S'}\in \sV_{S'}$ and at its fiber at $s'$, identifying $V = \omega_s({\bf V})$ and $\omega_{s'}({{\bf V}_{S'}})$. 

One has $V^{M_{S',s'}} = V^{M_{S,s}} \oplus V^\chi$ as $M_{S,s}$-representations. By Theorem \ref{t1A}, 
this is the fiber at $s'$ of a constant subvariation of ${\bf V}_{S'}$, and since its underlying local system comes from a local system over $S^{an}$, it is actually of the form ${\bf W}_{S'}$ for a subvariation of $\bf V$  (this follows from the definition of a variation of mixed Hodge structures). Hence $V^{M_{S,s}} \oplus V^\chi$ is stable under $T_{S,s}$, and we may substitute it to the original $V$; in other words, we may assume that $M_{S,s}$ acts by $\Z/2$ on $V$, and denoting by a subscript $\pm$ the $\pm 1$-eigenspaces, we have to show that $V_- $ is stable under $T_{S,s}$. 

Since 
$\Hom_{\Z/2}(V_+,V_-)=\Hom_{\Z/2}(V_-,V_+)=0,$
the $T_{S,s}$-equivariant exact sequence
\begin{equation}\label{e12.2}
0\to V_+\to V\to V/V_+\to 0
\end{equation}
has a unique $\Z/2$-equivariant splitting, and we have to show that this splitting is $T_{S,s}$-equivariant. For this, we apply \eqref{e12.2} to the dual representation $V^\vee$, and compare this exact sequence to the ($T_{S,s}$-equivariant) dual exact sequence of \eqref{e12.2} for $V$: their $\Z/2$-equivariant splittings both correspond to the decomposition $V^\vee=(V^\vee)_+\oplus (V^\vee)_-$. Therefore the ($T$-equivariant) morphisms
\[(V^\vee)_+ \to (V_+)^\vee,\quad (V/V_+)^\vee \to V^\vee/(V^\vee)_+ \]
are bijective and we get the claim for $V^\vee$, hence for $V$\footnote{this is the intrinsic version of the following shorter argument:

\noindent if $\underline{v}_\pm$ is a basis of $V_\pm$ with dual basis $\underline{v}_\pm^\vee$,  $(\underline{v}_+^\vee, \underline{v}_-^\vee)$ 
is the dual basis of the basis $(\underline{v}_+, \underline{v}_-)$ of $V$. 
 By \ref{t1A}, $T_{S,s}$ acts on $V$ by upper block-triangular matrices $T$, and on $V^\vee$ by upper block-triangular matrices ${}^tT^{-1}$.
Therefore the $T$'s are block-diagonal.}. 
 \end{rk}

 \subsubsection{} Let $S'$ be a finite Galois \'etale covering of $S$, with group $\Gamma$, and $s'$ a point of $S'$ above $s$. 
 
 \smallskip 
 Then $\Gamma$ acts as a finite group of (anti)auto-equivalences of $\Loc_{S'}$ and of $\sV_{S'}$\footnote{and even by (anti)auto\-morphisms, if one replaces $\Loc_{S'}$ and $\sV_{S'}$ by $Rep_\Q \, \Pi_{S',s'}$ and $Rep_\Q \, T_{S',s'}$ respectively.}. The tannakian categories given by looking at objects and morphisms fixed by $\Gamma$ are equivalent to $\Loc_{S}$ and of $\sV_{S}$ respectively. In particular, taking sums of conjugates by elements of $\Gamma$, one sees by ${\bf P}2$ that $\Pi_{S',s'}  \to \Pi_{S,s}$ and $T_{S',s'}  \to T_{S,s}$ are closed embeddings. 
 One has an exact sequence  \begin{equation}\label{e4A}1\to \Pi_{S',s'} \to \Pi_{S,s} \to \Gamma\to 1. \end{equation}

On the other hand, any representation $V$ of $\Gamma$ admits a $\Gamma$-invariant lattice and, by averaging some scalar product, carries a $\Gamma$-invariant scalar product. This allows one to attach to $V$ an object\footnote{which could be called an Artin object in $\sV_S$. When $S'$ varies, Artin objects form a tannakian subcategory of $\sV_S$ equivalent to ${\rm{Rep}}_\Q \pi_1^{et}(S,s)$.} $\bf V$ of $\sV_S$ with monodromy $\Gamma$ and whose fibers are trivial Hodge structures; whence a morphism $T_{S,s} \stackrel{p}{\to} \Gamma$ which, by ${\bf P}1$, is faithfully flat.  
 
 \begin{prop} The sequence of affine group schemes over $\Q$
\begin{equation}\label{e5A}
1\to T_{S',s'}\stackrel{q}{\to} T_{S,s} \stackrel{p}{\to} \Gamma \to 1
 \end{equation} is exact.
\end{prop}
 
 \begin{proof} Let $\bf V $ be an object of $\sV_S$, with fiber $V$ at $s$. Then $V^{T_{S',s'}}$ corresponds to the maximal constant subobject ${\bf W}'$ of ${\bf V}_{S'}$. By the exact sequence \eqref{e4A}, the local system  underlying  ${\bf W}'$ comes from a local system on $S^{an}$, hence is of the form ${\bf W}_{S'}$ for a subvariation $\bf W$ of $\bf V$ -- which becomes constant and even trivial after pulling back to $S'$.  Such an object belongs to $p^\ast {\rm{Rep}}_\Q\,  \Gamma$. This gives ${\bf H}0$ and ${\bf H}3$ for the sequence \eqref{e5A}. ${\bf H}5$ follows from Lemma \ref{l2A}, and one concludes by Proposition \ref{p4A}.
  \end{proof}

  \enlargethispage*{20pt}
 
Combining the compatible exact sequences \eqref{e4A} and \eqref{e5A}  with the connectedness of $\MT$ \cite[Lemma 2]{andrecomp}, one gets:
 
  \begin{cor}\label{C.15}  $ \pi_0(M_{S,s})=  \pi_0(T_{S,s})= \pi_1^{et}(S,s) $.\qed
 \end{cor}
 
  \subsubsection{} The first part of Lemma \ref{l2A} can be generalized as follows (Deligne \cite[4.2.9b]{HII}):
  
    \begin{cor}  The radical of $M_{S,s}^{0}$ is prounipotent (equivalently: characters of $M_{S,s\,\C}$ are of finite order). 
 \end{cor} 
 
 \begin{proof}(sketch) We start with the extension of \ref{l2A} to real coefficients: 
 
\smallskip\item{\it Claim: characters $\chi$ of $(M_{S,s})_\R$ occurring in a representation $V$ of $(T_{S,s})_\R$ are of finite order} (hence trivial or quadratic since $\R^\times_{tors} = \{\pm 1\}$).

Let $G$ (\resp $H$) be the image of $T_{S,s\,\R}$ in $GL(V)$. Replacing $S$ by a finite \'etale covering, we may assume that $H$ is connected, hence also $G$ (by Corollary \ref{C.15}). By Theorem \ref{t2A}, $H$ is normal in $G$, and by the converse of part of Proposition \ref{p1A}, $V^\chi$ is $G$-stable, hence is the fiber at $s$ of a graded-polarizable variation of real mixed Hodge structures; passing to the maximal exterior power, the polarization forces $\chi$ to be of finite order. This proves the claim.

 \medskip Let us now prove the corollary. The statement follows from the analogous statement about finite-dimensional quotients $M$ of $M_{S,s}^0$: namely, that the reductive quotient $M^{red}$ (\ie the quotient by its unipotent radical) is semisimple. Replacing $S$ by a finite etale covering, we may assume that $M^{red}$ is a quotient of $M_{S,s}$. It thus suffices to show that for any abelian semisimple representation $W$ of $M_{S,s}$, the image $A\subset GL(W)$ of $M_{S,s}$ is finite. Since $M_{S,s}$ is observable in $T_{S,s}$, $W$ extends to a representation $V$ of $T_{S,s}$. Changing notation, let $G $ (\resp $H$) be the image of $T_{S,s}$ (\resp $M_{S,s}$) in $GL(V)$. Then $A$ is an abelian reductive quotient of $H$. 
Due to the underlying $\Z$-structure, the image of $\pi_1(S,s)$ in $GL(W_\R)$ is discrete, and it suffices to shows that $A(\R)$ is compact - equivalently: every character of $A_\R$ is of finite order -, which follows from the claim. 
    \end{proof}

\end{document}